\documentclass[10pt,a4paper]{article}
\usepackage{amssymb,latexsym,amsmath,amsthm,authblk,mathrsfs,amsfonts}
\usepackage{fullpage}
\numberwithin{equation}{section}

\usepackage{blindtext}

\usepackage[mathscr]{euscript}
\usepackage[scr]{rsfso}

\usepackage{hyperref}

\newtheorem{theorem}{Theorem}
\newtheorem{proposition}{Proposition}
\newtheorem{lemma}{Lemma}
\newtheorem{corollary}{Corollary}

\theoremstyle{definition}
\newtheorem{definition}{Definition}
\newtheorem{definitions}[definition]{Definitions}

\newtheorem*{remark*}{Remark}
\newtheorem{example}[definition]{Example}
\newtheorem*{example*}{Example}

\newtheorem*{note}{Note}
\newtheorem*{notation}{Notation}
\newtheorem*{terminology}{Terminology}

\newtheorem*{ackn}{Acknowledgment}

\usepackage{xcolor}

\usepackage[normalem]{ulem}

\usepackage{enumitem}

\usepackage{aligned-overset}

\usepackage{blindtext}

\usepackage[margin=1.69in]{geometry}

\emergencystretch=\maxdimen
\begin{document}
		
\title{\textbf{Joint invariant sets for non-commutative expanding Markov maps of the circle}} 
		
\author{Georgios Lamprinakis}

\date{}

\maketitle
		
\begin{abstract}
 A long-standing question is what invariant sets can be shared by two maps acting on the same space. A similar question stands for invariant measures. A particular interesting case are expanding Markov maps of the circle. If the two involved maps are commuting the answer is almost complete. However very little is known in the non-commutative case. A first step is to analyze the structure of the invariant sets of a single map.
For a mapping of the circle of class $C^{\alpha}$, $\alpha>1$, we study the topological structure of the set containing all compact invariant sets. Furthermore for a fixed such mapping we examine  locally, in the category sense, how big is the subset of all  maps that have at least one non trivial joint invariant compact set. Lastly we show the strong dimensional relation of the maximal invariant set  of a given Markov map contained in a subinterval of $[0,1)$\footnote{Here by $[0,1)$ we denote is the unit interval where $0$ and $1$ are identified. The topology on $[0, 1)$ is the one that has as a base the intervals $(x - \epsilon, x + \epsilon)\cap [0, 1)$, if $x \in (0, 1)$ and $[0, \epsilon)\cup (1 -\epsilon, 1),$ if $x = 0$, where $\epsilon > 0$. A metric that is compatible with the topology in $[0,1)$ is $d_{[0,1)}(t, s) := \min\{|t - s|, 1 - |t - s|\}$, for $t, s \in [0, 1)$ and $([0,1), d_{[0,1)})$ is a compact metric space.} and the set of all right endpoints of its invariants sets contained in the same subinterval as well as the continuity dependence of the dimension on the endpoints of the subinterval.
\end{abstract}

\section{Introduction}

One long-standing question is which invariant sets can have two maps of the same space in common. In particular the case when the space is the unit interval or the circle and the maps are expanding Markov maps is interesting. One of the main difficulties here is that those maps are not globally invertible. Therefore (and for other reasons) the methods that helped to give answers in the setting of Anosov diffeomorphisms on the two dimensional torus, used by A. Brown and F.R. Hertz in \cite{Brown Hertz}, cannot be applied. On the other hand, A. Johnson and D. Rudolph showed that any two commuting Markov maps can be linearized simultaneously \cite{Johnson Rudolph}. That reduces the question to joined invariant sets of linear Markov maps that was solved by H. Furstenberg \cite{Furstenberg upbox dim = Haus dim = entropy}. We will consider the general case at least for a majority of maps. For this we study the structure of the set of invariant sets of a given set. A global result for two generic maps needs a restriction in terms of the Hausdorff dimension of the invariant sets while a local result does not need this restriction.

Let $f:[0,1) \to [0,1)$ be an expanding Markov map of the circle, i.e. there exist finitely many points $x_i\in [0,1)$ such that $f(x_{i-1},x_i)=[0,1)$, $f\in C^\alpha(x_{i-1},x_i)$ and a $\gamma > 1$ so that $|f'(x)|> \gamma$ for all $x\in [0,1)$. For such a map we consider $\mathcal{K}_f$ to be the set of all compact subsets of $[0,1)$ that are also invariant under the action of $f$. We endow the set $\mathcal{K}_f$ with the well known Hausdorff metric $d_H$. M. Urbanski in \cite{Urbanski87} and C.C. Conley in \cite{Conley} present some the topological properties of the this metric space for $C^2$ expanding Markov maps and for flows respectively. Motivated by that, we further study the topological structure of the metric space $(\mathcal{K}_f, d_H)$. More precisely, we will show that $(\mathcal{K}_f, d_H)$ is a compact and totally disconnected metric space.

Consider the set $\mathcal{E}^{\alpha}$, $\alpha \in (1, +\infty]$, of all $C^{\alpha}$ expanding Markov maps of the circle.  $\mathcal{E}^{\alpha}$ is endowed with the $\|\cdot\|_{C^{\alpha}}$ norm (the usual norm that is considered in $C^{\alpha}$). Given a function $f$ in $\mathcal{E}^{\alpha}$
we will show that for a generic function $g$ in $\mathcal{E}^{\alpha}$ and for all $K\in \mathcal{K}_f$, with sufficiently small Hausdorff dimension, $K\not\in \mathcal{K}_g$. As a matter of  fact we will show something even stronger, namely, for generic $g$ and for all $K$ with sufficiently small dimension, the intersection between $g(K)$ and $K$ is empty. A similar result without restriction on the dimension can be shown locally. More specifically, there is a open neighborhood where $f$ is contained in its closure so that for any $g$ in that neighborhood, there is no joint invariant set. 
It is worth noting that the case of $C^1$ functions was studied and fully resolved by C.G. Moreira in \cite{Moreira}.  

In the last part of this paper we give some dimensional results concerning the relation between the Hausdorff dimension of the largest $f$-invariant set contained in $[0,c] \subset [0,1)$ and of the "right endpoints" contained in the same subinterval, in an attempt to weaken the restriction, we further study the invariant sets by their endpoints. More precisely, if $f$ is in $\mathcal{E}^{\alpha}$ then for $c\in [0,1)$, we consider the sets $M_c'$ to be the set of all points for which their orbit remains below $c$ and $M_c$ to be the set of all points $x$ that their orbit stays in $[0,c)$ and $f^n(x)< x $ for all $n>1$. In \cite{Johan BAN} J. Nilsson shows that that for $f= D$, where $D$ is the doubling map acting on the circle then the sets $M_c$ and $M_c'$ have in fact the same Hausdorff dimension. We will prove that this is true, not only for the doubling map, but for every expanding Markov map in $\mathcal{E}^{\alpha}$, $\alpha \in (1, +\infty]$.

In a more general setting, we can show a similar result for the case we have both ways restrictions. Namely, let $M_{c,d}'$ denote the set of all points so that their orbit stays in an interval $[c,d]\subset [0,1)$. Respectively, $M_{c,d}$ is the set of all points $x$ so that not only their orbit remains in $[c,d]$ but also $f^n(x)<x$ for every $n>1$. This corresponds to right endpoints of their $\omega$-limit sets. Again we can prove that their respective Hausdorff dimensions in fact coincide, i.e. $\dim_H (M_{c,d}) = \dim_H(M_{c,d}')$.

Finally, we study the behaviour of the dimension of $M_{c,d}'$ as $c$, $d$ change. In \cite{Urbanski87} M. Urbanski also  examines the behaviour of the map $(c,d) \mapsto \dim_H(M_{c,d})$ but for the case of $C^2$ expanding Markov maps. In \cite{Johan BAN} J. Nilsson examines the behaviour of the map $c \mapsto \dim_H(M_c)$ for the doubling map using more elementary combinatorial methods. Our aim is to extend those results for any $f \in \mathcal{E}^{\alpha}$. In fact, we will show that the Hausdorff dimension of $M_{c,d}$ depends continuously on $(c,d)$, i.e. the map $(c,d) \mapsto \dim_H(M_{c,d})$ is continuous.

\smallskip

\section{Preliminaries} 

\subsection{Markov partition and Coding}
Let $\Sigma_m$ denote the full shift space corresponding to alphabet $\{ 0,1, \ldots , m-1 \}$, i.e. $\Sigma_m := \{ 0,1, \ldots , m-1 \}^{\mathbb{N}}$. The topology on $\Sigma_m$ is the product topology and it is a compact, metrizable topological space. Let $\sigma$ be the regular shift operator on $\Sigma_m$, such that $\big(\sigma(\underline{x})\big)_i = x_{i+1}$, for all $\underline{x} = (x_1, x_2, \ldots )\in \Sigma_m$. The shift operator act continuously on $\Sigma_m$. We also endow $\Sigma_m$ with the lexicographic order, i.e. if $\underline{a}$, $\underline{b} \in \Sigma_{m}$, then $\underline{a} < \underline{b}$ iff  there exists $i_0\in \mathbb{N}$ such that $a_i=b_i$ for all $1\leq i<i_0$ and $a_{i_0}< b_{i_0}$. For $a\in \{ 0,1, \ldots , m-1 \}$, we consider the cylinder set $C_a:=\{\underline{x}\in \Sigma_{m}: \ x_1=a\}$. The cylinder sets are both closed and open subsets of $\Sigma_m$. 

Let $\Sigma$ be a closed subset of $\Sigma_m$. Then $\Sigma$ is a subshift of finite type (SFT) if it is a closed under the shift operator subset of $\Sigma_m$ so that the forbidden blocks that describe $\Sigma$ consist a finite set. Of course the whole shift space is a SFT. A forbidden block $w=[w_1 \ldots  w_{\ell}]$ can also be described as the collection of larger blocks
	$$\{ [w_1 \ldots  w_{\ell}0], [w_1 \ldots  w_{\ell}1], \ldots , [w_1 \ldots  w_{\ell}(d-1)] \} \ .$$
Thus we can assume if needed, that all the forbidden words are of the same length, equal to that of the longest forbidden block.	
Any subshift of finite type can also be represented from a $m^{\ell-1}\times m^{\ell-1}$ matrix, $A=(a_{ij})$, with entries in $\{0,1\}$, where $\ell$ is the length of the longest forbidden word and $a_{ij} = 1$ when it corresponds to an allowed block and $a_{ij}=0$ otherwise.

Let $f:[0,1) \to [0,1)$ be an expanding Markov map of the circle, i.e. $f$ is a local homeomorphism and there exist finitely many points $x_i\in [0,1)$ such that $f(x_{i-1},x_i)=[0,1)$, $f\in C^\alpha(x_{i-1},x_i)$ and a $\gamma > 1$ so that $|f'(x)|> \gamma$ for all $x\in [0,1)$. We call the intervals $I_i=[x_i, x_{i+1}]$, fundamental intervals.

A Markov partition for an expanding Markov map $f$ is a finite cover $\mathcal{P} = \{P_1 , \ldots , P_d \}$ of $[0,1)$ such that,
	\begin{enumerate}
		\item each $P_i$ is the closure of its interior, int$P_i$
		\item int$P_i \cap \text{int} P_j = \emptyset$
		\item each $f(P_i)$ is a union of elements in $\mathcal{P}$.
	\end{enumerate} 

An expanding Markov map has Markov partition of arbitrary small diameter. If $\mathcal{P} = \{P_1 , \ldots , P_d \}$ is a Markov partition then $([0,1), f)$ can be represented, in a natural way, by a subshift of finite type, $\Sigma_A$ in $\Sigma_{d-1}$ so corresponding to the transfer matrix $A=(a_{ij})$, where 
$$a_{ij} = \begin{cases}
			1, \quad &\text{int}P_i \cap f^{-1}( \text{int}P_j) \neq \emptyset\\
			0, &\text{otherwise}
			\end{cases}$$ 
This gives a coding map $\chi: \Sigma_A \to [0,1) $, so that $\chi_f \circ \sigma = f\circ \chi_f$. Furthermore $\chi_f$ is H\"older continuous and injective on the set of points whose trajectory never hit the boundary of any $P_i$. These points are at most countable (see \cite{Pesin}).

Now let $f, g$ be two expanding Markov maps of class $C^{\alpha}$, $\alpha >1$. If $f, g$ have the same number of fundamental intervals, $\{ I_1, \ldots , I_d \}$, $\{ J_1, \ldots , J_d \}$, then they are topologically conjugated via a H\"older continuous homeomorphism, $h$, induced by the respective coding maps of the same coding space. Indeed, we consider the respective coding maps $\chi_f: \Sigma_d \to [0,1)$ and $\chi_g: \Sigma_d \to [0,1)$. Define for $x\in [0,1)$, $h(x):= \chi_g \big( \chi_f^{-1}(x) \big)$. Even if $x$ is a boundary point for some $I_i$ and $I_{i+1}$ and thus it can be represented by two sequences, one that ends with infinite $(i-1)$'s and one that ends with infinite $i$'s, then $\chi_g \big( \chi_f^{-1}(x) \big)$ is again a single point.  Then $h$ has all the requested properties.
For more details and proofs one can see for example \cite{Denker Grillenberger Sigmund, Lind Marcus, Pesin}.

\smallskip

\subsection{Entropy and Dimension}
Let $D$ denote the doubling map of the circle, $\Sigma_2 = \{ 0,1\}^{\mathbb{N}}$, endowed with the metric
$$d_2(\underline{a}, \underline{b}) := \sum_{i=1}^{\infty} \frac{|a_i - b_i|}{2^i} \ $$
and $\sigma$ be the shift map in $\Sigma_2$.
Then, as mentioned above, we can associate the system $([0,1),D)$ with with the space $(\Sigma_2, \sigma)$ with an almost one to one corresponding. The correspondence here is rather natural as we relate the sequence $(x_1, x_2, \ldots)$,  $x_i\in \{0,1\}$ with the real number, in $[0,1]$, $x = \sum_{i=1}^{\infty} \frac{x_i}{2^i}$.

\begin{theorem} \label{Th. upper box dim = Hausdorff dim = entropy}
	Let $A$ denote a compact invariant set of $(\Sigma_2, \sigma)$ and $A^*$ the corresponding invariant set on $([0,1], D)$. Then,
	$$\dim_H (A^*) = \dim_{box}(A^*) = \frac{\textup{h}_{\textup{top}}(A^*)}{\log 2}$$
\end{theorem}
\begin{proof}
	See \cite[Proposition III.1]{Furstenberg upbox dim = Haus dim = entropy}
\end{proof}

	Note that every invariant set in $([0,1], D)$ can be represented as an invariant set of $(\Sigma_2, \sigma)$,  through the onto map $h:\Sigma_2 \to [0,1]$ (coding map). Therefore Theorem \ref{Th. upper box dim = Hausdorff dim = entropy} tells us that when it comes down to $D$-invariant sets on $[0,1]$ the notions upper box-counting dimension, lower box-counting dimension, Hausdorff dimension and topological entropy coincide.

	From now on we will use the above result without any special mention. In other words, we will use the triple equality $\dim_H (A^*) = \dim_{box}(A^*) \thickapprox \textup{h}_\textup{top}(A^*)$ regularly without saying.

\begin{proposition} \label{Pr. dim + dim < 1}  
	Let $K_1$, $K_2 \subset \mathbb{R}$ with upper box-counting dimension $d_1$, $d_2$ respectively. Assume that $d_1+d_2<1$, then the Lebesgue measure of $K_1-K_2$ is zero.
\end{proposition}
\begin{proof}
	See \cite[Proposition 1, Chapter 4]{Palis Takens}, Note that the notion of Cantor set is not essentially used for the proof of this statement.
\end{proof}

Proposition \ref{Pr. dim + dim < 1} shows the relation between dimension and the possibility for any perturbation of one set to intersect with the other set. Namely, observe that $K_1-K_2$ is the set of all $t\in \mathbb{R}$ so that $K_1 \cap (K_2+t) \neq \emptyset$. The result above shows that almost surely this intersection is in fact empty, given that the two sets have sufficiently small dimension.

A rather straightforward result is the following Lemma. One shall only use the definition of the Hausdorff dimension and the Lipschitz continuity to prove it. A detailed proof as well as some more generally stated results can be found in the first Chapters of \cite{Pesin}.

\begin{lemma} \label{Lemma dim. invariant under Lip.}  
	Let $X$, $Y$ metric spaces, $A\subset X$ and $f: X \to Y$ Lipschitz continuous. Then $\dim_H(f(A)) \leq \dim_H(A)$
\end{lemma}

\begin{lemma}  \label{Lemma closed map lemma}  
Let $X$ be a compact topological space and $Y$ be a Hausdorff topological space. If \mbox{$f:X\to Y$} is a continuous function, then $f$ is closed and proper. 
\end{lemma}

Another important fact that will be used extensively throughout this paper is the semicontinuity of the entropy.

\begin{lemma}(Semicontinuity of entropy) \label{Lemma semicontinuity of entropy}
Let $X_n$ be subshifts (not necessarily of finite type) and assume that \hbox{$X_{n+1} \subset X_{n}$}. If $X=\bigcap_{n\in \mathbb{N}}X_n$, then 
		$$h_{top}(X) = \lim_{n\to \infty}h_{top}(X_n)$$
\end{lemma}
\begin{proof}
		The limit indeed exists and $\lim_{n\to \infty}h_{top}(X_n)=\inf_{n\in \mathbb{N}}h_{top}(X_n)$, since \mbox{$X_{n+1} \subset X_{n}$} implies that the sequence $\{h_{top}(X_n)\}_{n\in \mathbb{N}}$ is a decreasing sequence in $\mathbb{R}$. We obviously have that $h_{top}(X)\leq h_{top}(X_n)$ for every $n\in \mathbb{N}$. Thus, $$h_{top}(X) \leq \lim_{n\to \infty}h_{top}(X_n).$$
		Now for every $X_n$ there exists a measure of maximal entropy, $\mu_n$, such that $h_{\mu_n}(X_n)=h_{top}(X_n)$ (see for example \cite[Chapter 17, Corollary 2, pp 130]{Denker Grillenberger Sigmund}). Any weak-$\ast$ accumulation point $\mu$ of the sequence $\mu_n$ is supported by $X$. Therefore, 
		$$h_{\mu}(X)=h_{\mu}(X_n) \ ,$$ 
		for every $n\in \mathbb{N}$. By the  upper semicountuity of (measure theoretic) entropy in subshifts, i.e. of the map $\mu \mapsto h_{\mu}(\Sigma)$, where $\Sigma$ is a subshift, we have that for every $n_0\in \mathbb{N}$,
		$$h_{\mu}(X)=h_{\mu}(X_{n_0})\ge \limsup_{n\to \infty} h_{\mu_n}(X_{n_0}).$$
		By the nested property and since each $\mu_n$ is supported by $X_n$, for $n\geq n_0$, $h_{\mu_{n}}(X_{n_0}) = h_{\mu_n}(X_n)$. Thus,
		$$h_{\mu}(X)\ge \limsup_{n\to \infty} h_{\mu_n}(X_{n}).$$
		Finally, since $$h_{top}(X)=\sup\{h_{\mu}(X) \ \big| \ \mu \textup{ is a } \sigma \textup{-invariant  Borel probability measure on }X\} \ ,$$ we have that,
		$$h_{top}(X)\ge h_{\mu}(X)\ge \limsup h_{\mu_n}(X_n) = \lim_n h_{top}(X_n).$$
\end{proof}

\begin{lemma} \label{Lemma  density of SFT from above} 
If $K\in \mathcal{K}_f$, $f\in \mathcal{E}^{\alpha}$  and $K\neq SFT$ then, $K=\bigcap_{n\in \mathbb{N}}SFT_n$, where \hbox{$SFT_{n+1}\subset SFT_n$}, for all $n\in \mathbb{N}$.
\end{lemma}
\begin{proof}
Let $K\in \mathcal{K}_{f}$. Set $\tilde{K}: =h^{-1}(K)$, where $h:(\Sigma_d, \sigma) \to ([0,1), f)$ is the respective coding map. Then $\tilde{K}$ is a compact subshift of $\Sigma_d$ and let $L_{\tilde{K}}:=\{w_1, w_2, \dots \}$ be the (countable) set of all the forbidden blocks for $\tilde{K}$. If $L_{\tilde{K}}$ is finite then $\tilde{K}$ is a subshift of finite type. If not, we consider for each $n\in \mathbb{N}$, $L_n:=\{w_1, w_2, \ldots , w_n\}$ and we consider the respective subshift of finite type, $SFT_n$. Then clearly \mbox{$SFT_{n+1}\subset SFT_n$} and $\tilde{K}=\cap_{n\geq 1}SFT_n$. By the nesting property of $\{SFT_n\}$ and since every $x\in [0,1]$ has a finite number of $h$-preimages in $\Sigma_d^+$, we have that $K = \cap_{n\geq 1} S_n^{\ast}$, where $S_n^{\ast}=h(SFT_n)$.
\end{proof}

By Lemma \ref{Lemma semicontinuity of entropy} and Lemma \ref{Lemma  density of SFT from above} we can deduce immediately the following corollary.

\begin{corollary} \label{Cor. d(SFT) approaches d(K)} 
	Let $K\in \mathcal{K}_D$. For every $\varepsilon \ge 0$, $\exists SFT$ such that $K\subset SFT$ and \mbox{$d(SFT)-\varepsilon < d(K)$}.
	In particular, if $d(K)< \frac{1}{2}$, $\exists SFT$ such that $K\subset SFT$ and \mbox{$d(SFT)<\frac{1}{2}$}.
\end{corollary}

\begin{note}
	Observe that, even if not stated as generally as possible, all the results hold for every symbolic system and thus, not only for the doubling map, but also for any expanding system of the circle. Only the general notion of the shift spaces and subshifts of finite type were used, as well as the corresponding coding mapping of each such system mentioned above.
\end{note}

\smallskip

Let $f\in \mathcal{E}^{\alpha}$, $\alpha>1$ and $\sigma: \Sigma_{m+1} \to \Sigma_{m+1}$ be its corresponding coding space, where $ \Sigma_{m+1}= \{ 0,1, \ldots , m\}^{\mathbb{N}}$. Instead of considering the sets $M_c$, $M_{c}'$ and $M_{c,d}$, $M_{c,d}'$, we consider the corresponding sets $M_{\underline{c}}$, $M_{\underline{c}}'$ and $M_{\underline{c},\underline{d}}$, $M_{\underline{c},\underline{d}}'$, where $\underline{c}$, $\underline{d} \in \Sigma_{m+1}$ and  $\underline{c} < \underline{d}$. Let $K$ be a compact $f$-invariant set. We consider the corresponding invariant set $S$ in $\Sigma_{m+1}$. The right endpoint, $\underline{x}_r$, of $S$, since it is invariant, has the property $\sigma^{n}(\underline{x}_r) \leq \underline{x}_r$, for all $n\in \mathbb{N}$. Furthermore $\underline{x}_r$ correspond to the right endpoint of $K$, $x_r$. 

Observe that the set if $\underline{x}$ is such that there is an $n$ so that $\sigma^{n}(\underline{x})=\underline{x}$, then for $n_0= \min \{n\in\mathbb{N}: \ \sigma^{n}(\underline{x})=\underline{x} \}$, we have that $\underline{x}$ is of the form,
$$\underline{x}=x_1x_2 \ldots x_{n_0} x_1x_2\ldots x_{n_0} x_1x_2 \ldots x_{n_0} \ \ldots \ = \ (x_1x_2 \ldots x_{n_0})^{\infty}$$
Therefore, if $B_n$ denotes a block of length $n$ from the alphabet $\{0,1, \ldots , m \}$, we have that,
$$\{\underline{x}\in \Sigma_{m+1} : \ \exists n_0=n_0(\underline{x}) \text{ such that } \sigma^{n_0}(\underline{x})=\underline{x}\} 
= 
\bigcup_{n=1}^{\infty} \{ (B_n)^{\infty}: \ B_n \text{ is an } n \text{-block}\}$$
where the last set is countable as countable union of finite sets.
Therefore the Hausdorff dimension is not affected if we exclude those points. In particular, for our dimensional related results it suffices to define $M_{\underline{c}}$ and $M_{\underline{c}, \underline{d}}$, with strict inequalities.

\begin{definitions}\hfill \label{definitions}
\begin{enumerate}
\item $M:=\{\underline{x}\in \Sigma_{m+1} : \ \sigma^n(\underline{x}) < \underline{x} , \ \forall n>0\}$.
\item $M_{\underline{c}}' := \{ \underline{x}\in \Sigma_{m+1} : \ \sigma^n(\underline{x}) \leq \underline{c} , \ \forall n\geq 0 \}$.
\item $M_{\underline{c}} := \{ \underline{x}\in \Sigma_{m+1} : \ \sigma^n(\underline{x})< \underline{x} \leq \underline{c} , \ \forall n> 0 \}$.
\item \label{reversing lexicograpfic} If $\underline{x}=(x_1,x_2,\ldots) \in \Sigma_{m+1}$ then, $\tilde{\underline{x}}:= (m-x_1,m-x_2,\ldots)$. If $A\subset \Sigma_{m+1}$ then, $\widetilde{A}:=\{ \underline{x}\in \Sigma_{m+1} : \ \tilde{\underline{x}} \in A \}$.
\end{enumerate}
\end{definitions}

\begin{example} \label{example of inversion}
\begin{align*}
\widetilde{M_{\underline{\widetilde{c}}}'} 
&= \ 
\{ \underline{x}\in \Sigma_{m+1} : \ \tilde{\underline{x}} \in M_{\underline{\widetilde{c}}}'\} 
= \ 
\{ \underline{x}\in \Sigma_{m+1} : \ \sigma^n(\tilde{\underline{x}}) \leq \underline{\widetilde{c}}  , \ \forall n\geq 0\} 
\\&= \ 
\{ \underline{x}\in \Sigma_{m+1} : \ \widetilde{\sigma^n(\underline{x})} \leq \underline{\widetilde{c}}  , \ \forall n\geq 0\} 
= \ 
\{ \underline{x}\in \Sigma_{m+1} : \ \sigma^n(\underline{x}) \geq c  , \ \forall n\geq 0 \}
\end{align*}
\end{example}

The definitions above are related to the compact $f$-invariant sets through the coding map. More specifically $M_{\underline{c}}'$ contains the largest invariant set that lies inside the subinterval $[\underline{0}, \underline{c}]$, since the orbit of any $\underline{x}$ in there does not escape this interval. From the discussion above, the set $M_{\underline{c}}$ contains all the right end points of all the invariant sets that are contained in $M_{\underline{c}}'$, up to a set of zero Hausdorff dimension.

\subsection{$\beta$-shift}
\begin{definition} 
Let $\beta \in (1, +\infty)$. Then, for every $x\in [0,1]$, the $\beta$-expansion or the expansion in base $\beta$ of $x$ is a sequence of integers out of $\{0,1,2,\ldots, [\beta]\}$, such that $x_n = [\beta T_{\beta}^{n-1}(x) ] $, where $T_{\beta}:[0,1)\to [0,1)$ is defined by $T_{\beta}(x):= \beta x$ (mod$1$).
\end{definition}

\begin{definition}
The closure of the set of all $\beta$-expansions of $x \in [0,1)$ is called the $\beta$-shift $S_{\beta}$.
\end{definition}

\begin{definition} 
The $\beta$-expansion of $1$, for some $\beta>1$, is denoted by $1_{\beta}$.
\end{definition}

W. Parry showed in \cite{Parry beta} that the $\beta$-shift, $S_{\beta}$, is completely characterised by its expansions of $1$.

\begin{proposition}\label{Parry 0} 
If the $\beta$-expansion of $1$ is $(\beta_1, \beta_2, \ldots )$ and the $\beta'$-expansion of $1$ is $(\beta_1', \beta_2', \ldots )$
then $\beta > \beta'$ if and only if $(\beta_1, \beta_2, \ldots ) > (\beta_1', \beta_2', \ldots )$.
\end{proposition}

\begin{theorem} \textup{(Parry)} \label{Parry 1} 
If $1_{\beta}$ is not finite, then an $\underline{s}\in \{0,1,2,\ldots, [\beta]\}^{\mathbb{N}}$ belongs to $S_{\beta}$ if and only if
$$\sigma^{n}(\underline{s}) < 1_{\beta}, \quad \forall n \geq 1$$
If $1_{\beta}$ is of the form $i_1i_2\ldots i_M0^{\infty}$, then $\underline{s} \in \{0,1,2,\ldots, [\beta]\}^{\mathbb{N}}$ belongs to $S_{\beta}$ if and only if,
$$\sigma^{n}(\underline{s}) < \left(i_1i_2\ldots i_{M-1}(i_M-1)\right)^{\infty}, \quad \forall n \geq 1$$
\end{theorem}
 
\begin{theorem} \textup{(Parry)} \label{Parry 2} 
A sequence $\underline{s}\in \{0,1,2,\ldots, [\beta]\}^{\mathbb{N}}$ is an expansion of $1$ for some $\beta$ if and only if, 
$\sigma^n(\underline{s}) < \underline{s}, \quad \forall n\geq 1$
and then $\beta$ is unique. Moreover, the map $\Xi : \beta \mapsto 1_{\beta}$ is monotone increasing.
\end{theorem}

Let $1< \beta_1 < \beta_2$. Then all sequences of expansions of one, $1_{\beta}$, for $\beta_1 < \beta < \beta_2$ are at the same time expansions for some $x\in [0,1)$ in base $\beta_2$. Let us denote the set of those $x$'s by $I(\beta_1 , \beta_2)$. If $ \pi : [0, 1) \to S_{\beta_2}$ is the map assigning to each
$x \in [0, 1)$ its $\beta_2$-expansion we define the map, (depending on $\beta_1$ and $\beta_2$), $$\rho_{\beta_1, \beta_2} : I(\beta_1 , \beta_2 ) \to [\beta_1 , \beta_2 ]$$
by setting $\rho_{\beta_1, \beta_2}(x)$ to be the unique $\beta \in [\beta_1 , \beta_2 ]$ having $\pi(x)$ as its expansion of $1$, i.e. $\pi(x)=1_{\beta}$. 

In \cite{Jorg}, J.  Schmeling showed that the map $\rho_{\beta_1, \beta_2}$ is H\"older continuous and calculated the H\"older-exponent.

\begin{theorem} \textup{(Schmeling)} \label{Schmeling} 
The map $\rho = \rho_{\beta_1, \beta_2} : I(\beta_1 , \beta_2) \to [\beta_1, \beta_2]$ satisfies the H\"older condition, 
$$ |\rho(\underline{u}) - \rho(\underline{v})| \leq C \cdot d(\underline{u}, \underline{v})^{\ln \beta_1 / \ln \beta_2}$$
where $d$ is the metric on $\Sigma_{[\beta_2]}$.
\end{theorem}

\begin{definition}
The sequence $\underline{x} \in \Sigma_{m+1}$is called kneading if for all $n>1$, $\sigma^n(\underline{x}) < \underline{x}$.
\end{definition}

\begin{remark*}
By Theorem \ref{Parry 2} a kneading sequence corresponds to a $\beta$-expansion of $1$ for some $\beta \in (1,m]$. Therefore, $M_{\underline{c}}$ is in fact the set of all points in the interval $[\underline{0}, \underline{c}]$ that are expansion of $1$ for some $\beta\in (1,m]$.
\end{remark*}

\smallskip

\section{Structure of the set of D-invariant sets}
\begin{notation}
If $A$ is a subset of $[0,1)$ then we define \mbox{$\mathcal{U}_{\epsilon}^{[0,1)}(A):= \{x\in [0,1) : \ d(x, A) < \epsilon\}$.}
\end{notation}
\begin{lemma}(Maximality of SFT) \label{Lemma  SFT maximal}
If $\Sigma\in \mathcal{K}_D$ is a subshift of finite type, then there exists an $\epsilon_{\Sigma} > 0$ such that, if $K$ is a $D$-invariant set so that $K \subset \mathcal{U}_{\epsilon_{\Sigma}}^{[0,1)}(\Sigma)$, then $K \subset \Sigma$.
\end{lemma}

\begin{remark*}
If we consider the subshift of finite type $\Sigma$ as a subset of $\Sigma_d$ then it is rather straightforward to show the maximality property. Indeed, let $w_k = [w_1^k \ldots w_{\ell}^k]$, $w_i^k \in \{0, 1, \ldots , d-1\}$, $1\leq i \leq \ell$, be a forbidden block and assume that there are in total $n$ forbidden blocks. We also assume that the length of all forbidden blocks is the same and equal to $\ell$. Consider $C_k:= \{\underline{x} \in \Sigma_d: \ (x_1, \ldots, x_{\ell})=(w_1^k, \ldots , w_{\ell}^k) \} = C_{w_1^k} \cap \sigma^{-1}(C_{w_2^k}) \cap \ldots \cap \sigma^{-\ell+1}(C_{w_{\ell}^k}) $, $1\leq k \leq n$. Then, each $C_k$ is clopen since the cylinders are clopen and $\sigma$ continuous. Furthermore,
	$$\Sigma = \bigcap_{m=0}^{\infty} \bigcap_{k=1}^{n} \sigma^{-m}(\Sigma_d \setminus C_k)$$
Thus for $\mathcal{U}:= \bigcap_{k=1}^{n}(\Sigma_d \setminus C_k)$ we get the result.
\end{remark*}

\begin{proof}
Let $\Sigma\subset [0,1)$ be a subset of finite type for $f\in \mathcal{E}^{\alpha}$ corresponding to the Markov Partition of the interval $\mathcal{P}=\{P_1,\ldots , P_d\}$, related to $f$. Consider the derived closed intervals of $[0,1)$, $P_{\ell_1 \ell_2 \cdots \ell_r}:= P_{\ell_1}\cap f^{-1}(P_{\ell_2}) \cap \ldots \cap f^{-r+1}(P_{\ell_r})$, $\ell_i \in \{1,2, \ldots , d\}$, $\forall i\in \{1,2, \ldots , r\}$.
	
	Let $R_0\in \mathbb{N}$ denote the number of all forbidden words. Then we may assume that all the forbidden words have the same length, i.e. there exists $r_0\in \mathbb{N}$ such that $w_i=w_{i_1}\cdots w_{i_{r_0}}$, for all $i\in \{ 1,2, \ldots , R_0 \}$. 
	
	Now let $\epsilon_0 \ll \min \{ \text{diam}(P_{\ell_1 \ell_2 \cdots \ell_{r_0}}): \ \ell_1 \ell_2 \cdots \ell_{r_0} \text{  is not a forbidden word}\}$ and $\underline{\ell}= \ell_1 \ell_2 \cdots \ell_n$ be an allowed word. We refer to the previous, with respect to the lexicographic order, word as the left from $\underline{\ell} = \ell_1 \ell_2 \cdots \ell_n$ word. Respectively, we define the right word. For the word $0^{r_0}$, the left word is $d^{r_0}$ and for the case $d^{r_0}$ the right word is $0^{r_0}$. We write any fixed $P_{\underline{\ell}}$ in the form  $[a_{\underline{\ell}}, b_{\underline{\ell}}]$ and we set,
$$\mathcal{W}(P):= 
\begin{cases} 
	[a_{\underline{\ell}}, b_{\underline{\ell}}] \ , \quad &\text{left and right word from } \underline{\ell} \text{ are both allowed} \\
	(a_{\underline{\ell}} - \epsilon , b_{\underline{\ell}}] \ , \quad &\text{only the right word from } \underline{\ell} \text{ is allowed} \\
	[a_{\underline{\ell}} , b_{\underline{\ell}}+ \epsilon) \ , \quad &\text{only the left word from } \underline{\ell} \text{ is allowed}  \\ 
	(a_{\underline{\ell}}-\epsilon , b_{\underline{\ell}} + \epsilon) \ , \quad &\text{left and right word from } \underline{\ell} \text{ are both not allowed}
	\end{cases}$$	
	
	Now set,
	$$\mathcal{U}(\Sigma) := \bigcup\limits_{\substack{\ell_1 \ell_2 \cdots \ell_n \\ \text{is an} \\ \text{allowed word}}} \mathcal{W}(P_{\ell_1 \ell_2 \cdots \ell_n})$$
Then $\mathcal{U}(\Sigma)$ is clearly open and $\cap_{n\geq 0}f^{-n}(\mathcal{U}(\Sigma)) \supset \Sigma$. In fact, it has the requested property,
$$\bigcap_{n\geq 0} f^{-n}(\mathcal{U}(\Sigma)) = \Sigma \ .$$ 
Indeed, let us assume that there exists an $x \in \cap_{n\geq 0}f^{-n}(\mathcal{U}(\Sigma)) \setminus \Sigma$ or equivalently, $x \in f^{-n}(\mathcal{U}(\Sigma)) \setminus \Sigma$, for all $n\geq 0$. This means that, for $n=0$, there exists a boundary point of some $P_{\ell_1 \ldots \ell_{r_0}}$, $s_0$, where either the left or the right word from $\underline{\ell}=\ell_1 \ldots \ell_{r_0}$ is not allowed, so that $x\in (s_0- \epsilon_0 , s_0 + \epsilon_0) \setminus \Sigma$, where $\epsilon_0=\epsilon$. Also, since boundary points go to boundary points, there exists another boundary point, $t_0$, of some  $P_{\ell_1'\ldots \ell_{r_0}'}$, where $s_0\in f^{-1}(t_0)$, so that $x\in f^{-1}((t_0-\epsilon_0 , t_0 + \epsilon_0)) \setminus \Sigma$. In particular  $(s_0- \epsilon_0 , s_0 + \epsilon_0) \cap f^{-1}((t_0-\epsilon_0 , t_0 + \epsilon_0)) \setminus \Sigma \neq \emptyset$. Since $s_0\in f^{-1}(t_0)$ and $\epsilon_0=\epsilon$ is much smaller that the diameter of the partition $\mathcal{P}$, this intersection is contained in $(s_0-\epsilon_1 , s_0 + \epsilon_1)$, $\epsilon_1 < \epsilon_0$, since $f^{-1}$ contracts intervals. In the same manner we can find sequences $(t_n)_{n\geq 1}$ and $(\epsilon_n)_{n\geq 0}$ so that $\epsilon_n \searrow 0$ and
\begin{align*}
x\in \big((s_0- \epsilon_0 , s_0 + \epsilon_0) \cap f^{-1}((t_0-\epsilon_0 , t_0 + \epsilon_0)) \cap \ldots \cap f^{-n}&((t_n- \epsilon_0, t_n + \epsilon_0))\big)\setminus \Sigma
\\&\subset 
(s_0 -\epsilon_{n+1}, s_0 + \epsilon_{n+1})
\end{align*}
But $(s_0 -\epsilon_{n}, s_0 + \epsilon_{n}) \to \{s_0\}$, when $n\to \infty$. Thus either the intersection is empty or $x=s_0$ both of which are contradictions. 
\end{proof}

\begin{corollary} \label{Cor.  V_SFT clopen}
Let $\Sigma\in \mathcal{K}_D$ be a subshift of finite type and define $$\mathcal{V}_{\Sigma}:= \{ S\in \mathcal{K}_D : \ S \subseteq \Sigma \} \ .$$
Then $\mathcal{V}_{\Sigma}$ is a clopen subset of $(\mathcal{K}_D, d_H)$.
\end{corollary}
\begin{proof}
If $L \in \text{cl}_{d_H}(\mathcal{V}_{\Sigma})$, then there exist $S\in \mathcal{V}_{\Sigma}$ such that $d_{H}(S, L) < \epsilon_{\Sigma}$, where $\epsilon_{\Sigma}$ as in Lemma \ref{Lemma SFT maximal}. In particular, $L \subset \mathcal{U}_{\epsilon_{\Sigma}}^{[0,1)}(S) \subset \mathcal{U}_{\epsilon_{\Sigma}}^{[0,1)}(\Sigma)$. Follows that $L\subset \Sigma$ and thus $L \in \mathcal{V}_{\Sigma}$. Therefore $\mathcal{V}_{\Sigma}$ is a closed set.

The same argument shows that if $S\in \mathcal{V}_{\Sigma}$, then for any $L\in \mathcal{K}_D$ such that $d_{H}(S, L) < \epsilon_{\Sigma}$ we have that $L\in \mathcal{V}_{\Sigma}$. Therefore $\mathcal{V}_{\Sigma}$ is an open set.
\end{proof}

\begin{lemma} \label{Lemma  K cap S empty then local diconnentedness}
Let $K_1$, $K_2 \in \mathcal{K}_D$ such that $K_1 \cap K_2= \varnothing$. Then there exist subshifts of finite type, $\Sigma_1 \supset K_1$, $\Sigma_2\supset K_2$, such that $\mathcal{V}_{\Sigma_1} \cap \mathcal{V}_{\Sigma_2} = \varnothing$.
\end{lemma}
\begin{proof}
If $K_1$, $K_2 \in \mathcal{K}_D$ such that $K_1 \cap K_2= \varnothing$, then by Lemma \ref{Lemma  density of SFT from above} there exist subshifts of finite type $\Sigma_1 \supset K_1$, $\Sigma_2 \supset K_2$ such that $\Sigma_1 \cap \Sigma_2 = \varnothing$. Follows that \mbox{$\mathcal{V}_{\Sigma_1} \cap \mathcal{V}_{\Sigma_2} = \varnothing$.}
\end{proof}

\begin{theorem} (Disconnectedness of $\mathcal{K}_{D}$) \label{Pr. structure of set of D-invariant sets} 
The space $(\mathcal{K}_D, d_H)$ is compact and totally disconnected. In particular it is of first category.
\end{theorem}
\begin{proof}
Compactness can be easily derived by observing that $\mathcal{K}_D$ is a $d_H$-closed subset of $\mathcal{M}$.

Let $\mathcal{C}$ be the connected component that contains $K\in  \mathcal{K}_D$. Then $K$ is a maximal element of $\mathcal{C}$. Indeed, if $S \in \mathcal{C}$ such that $K\subsetneqq S$, by Lemma \ref{Lemma  density of SFT from above}, there exists a subshift of finite type, $\Sigma$, such that $K\subsetneqq \Sigma \subsetneqq S$. We then have that $K\in \mathcal{V}_{\Sigma}$ and $S\not\in \mathcal{V}_{\Sigma}$. Since $\mathcal{V}_{\Sigma}$ is clopen (Corollary \ref{Cor.  V_SFT clopen}) and $\mathcal{C}$ is the connected component that contains $K$, we have that $\mathcal{V}_{\Sigma}\supset \mathcal{C}$ (contradiction since $S\not\in \mathcal{V}_{\Sigma}$ and $S \in \mathcal{C}$).

Finally, let's assume that $\mathcal{C}$ contains at least two distinct elements, $K_1$ and $K_2$. Then, since both of them are maximal and distinct, $K_1\cap K_2 = \varnothing$. By Lemma \ref{Lemma  K cap S empty then local diconnentedness}, there is a subshift of finite type $\Sigma$ such that $K_1\in \mathcal{V}_{\Sigma}$ and $K_2 \not\in \mathcal{V}_{\Sigma}$. With the same argument as before we get a contradiction and that completes the proof.
\end{proof}

\begin{note}
All the results above can be similarly be stated and proven when instead of the doubling map, $D$, we consider any map $f\in \mathcal{E}^{\alpha}$.
\end{note}

\smallskip

\section{Global result for subsets of ``small" Hausdorff dimension}
The following proposition is crucial to show that there is a $\| \cdot \|_{C^{\alpha}}$-residual set in $\mathcal{E}^{\alpha}$ so that any $f$ in that set has no joint invariant compact invariant set, of "small" Hausdorff dimension. This dimensional restriction arise from Proposition \ref{Pr. dim + dim < 1}, which is essential in order to acquire a global result for all  subshifts of finite type with dimension less that $1/2$. The strong sense of denseness of subshifts of finite type from Lemma \ref{Lemma  density of SFT from above}, allows us to pass to all compact invariant sets with sufficiently small dimension.

\begin{proposition} \label{Pr. dimension result 1, less than 1/2} 
For $\| \cdot \|_{C^{\alpha}}$-generic $f\in \mathcal{E}^{\alpha}$ and for every $SFT$ with \hbox{$d(SFT)<\frac{1}{2}$,}
$$f(SFT) \cap SFT = \varnothing \ .$$
\end{proposition}
\begin{proof}
Let $g\in \mathcal{E}^{\alpha}$ and $SFT$ such that $g(SFT)\cap SFT \neq \varnothing$.
Since $g\in \mathcal{E}^{\alpha}$, $\alpha>1$, then $g$ is Lipschitz continuous. Thus, by Lemma \ref{Lemma dim. invariant under Lip.}, $\dim_H (g(SFT)) \leq \dim_H (SFT) < \frac{1}{2}$.
By Proposition \ref{Pr. dim + dim < 1} we have that $\lambda \big(SFT - g(SFT)\big)=0$. Since $SFT - g(SFT) = \{ t\in [0,1]: \ SFT \cap (g(SFT)+t) \neq \emptyset \}$, it follows that there exist arbitrary small $\varepsilon >0$  such that $\big(g(SFT)+\varepsilon\big) \cap SFT = \varnothing$. 

Thus, for arbitrarily small $\varepsilon >0$ as above, by defining $g_{\varepsilon}=g+ \varepsilon$ we have that $g_{\varepsilon}\in \mathcal{E}^{\alpha}$, $\|g_{\varepsilon} - g\|_{C^{\alpha}} = \|g_{\varepsilon} - g\|_{\infty} = \varepsilon$ and $g_{\varepsilon}(SFT) \cap SFT = \varnothing$.

We have that the set  $\{SFT : \ \dim_H(SFT)< \frac{1}{2} \}$ is countable. Let $\{SFT_n\}_{n\in \mathbb{N}}$ be an enumeration of this set and define $\mathcal{G}_n=\{f\in \mathcal{E}^{\alpha} : \ f(SFT_n) \cap SFT_n = \varnothing \}$. Then from what discussed above $\mathcal{G}_n$ is $\| \cdot \|_{C^{\alpha}}$-dense in $\mathcal{E}^{\alpha}$, $\forall n\in \mathbb{N}$.

By Lemma \ref{Lemma closed map lemma}, if $f\in \mathcal{E}^{\alpha}$ then $f$ is a closed map. Hence, if $f\in \mathcal{G}_n$ then \mbox{$f(SFT_n) \cap SFT_n = \varnothing$} and since both sets are compact, $\exists \delta =\delta_{n,f}>0$ s.t. $dist(f(SFT_n), SFT_n) \geq \delta$. Thus the ball $B_{C^{\alpha}}(f, \frac{\delta}{2}) \subset \mathcal{G}_n$. Therefore $\mathcal{G}_n$ is $\| \cdot \|_{C^{\alpha}}$-open in $\mathcal{E}^{\alpha}$, $\forall n \in \mathbb{N}$. In particular, it is $\| \cdot \|_{C^0}$-open.

Baire's Theorem indicates that $\mathcal{G}:= \bigcap_{n\in \mathbb{N}} \mathcal{G}_n$ is $\| \cdot \|_{C^{\alpha}}$-dense in $\mathcal{E}^{\alpha}$.  
\end{proof}

\begin{remark*}
The method used in the proof of Proposition \ref{Pr. dimension result 1, less than 1/2} can be used to show the same result for any countable set of compact $f$-invariant subsets of ``small" Hausdorff dimension. Moreover, arguing in a similar way, a weaker result can be shown for any countable set, $\mathcal{A}$, of compact $D$-invariant subsets of arbitrary Hausdorff dimension.  Namely, for $\| \cdot \|_{C^{\alpha}}$-generic $f\in \mathcal{E}^{\alpha}$ and for every $K \in \mathcal{A} \subset \mathcal{K}_D$, $f(K) \neq K$.
\end{remark*}

\begin{theorem} \label{Th. dimension result 1'} 
For $\| \cdot \|_{C^{\alpha}}$-generic $f\in \mathcal{E}^{\alpha}$ and for every $K\in \mathcal{K}_D$ with \hbox{$d(K)<\frac{1}{2}$,}
	$$f(K) \cap K = \varnothing \ .$$
\end{theorem}

\begin{note}
All the results above can be similarly be stated and proven when instead of the doubling map, $D$, we consider any map $f\in \mathcal{E}^{\alpha}$.
\end{note}

\smallskip

\section{Local result for all subsets}

Here we show that locally, i.e. for an open neighbohood in $\mathcal{E}^{\alpha}$ such that its closure contains $D$, any other map has no joint invariant compact sets. It is a local result in the sense that the density of the open neighborhood, as in Proposition \ref{Pr. dimension result 1, less than 1/2}, in all $\mathcal{E}^{\alpha}$ is now replaced by the weaker condition that $D$  is in its closure. That being said, the restrictions in the dimension are lifted.

\begin{theorem} \label{Th. locally result} 
There is an $\| \cdot \|_{C^{\alpha}}$-open neighborhood in $\mathcal{E}^{\alpha}$, $G_{D}$, such that for any $f \in G_{D}$, $f(K)\neq K$, $\forall K\in \mathcal{K}_D$ and $D \in \text{cl}(G_D)$.
\end{theorem}
\begin{proof}
	Set $G_D:=\{ f\in \mathcal{E}^{\alpha} : \ f(K)\neq K, \ \forall K\in \mathcal{K}_D \}$ and let $\epsilon > 0$, $\epsilon \in \mathbb{R}\setminus \mathbb{Q}$. Then, $D_{\epsilon} := D + \epsilon \ (mod1)$, $D_{\epsilon}(K)\neq K$, $\forall K\in \mathcal{K}_D$. Indeed, if $K\in \mathcal{K}_D$,  $D_{\epsilon}(K)=K$ is equivalent to $K+\epsilon \ (mod1) = K$. Since $T_{\epsilon}(x):= x+ \epsilon (mod1)$ is minimal, this is only true for $K=\varnothing$ or $K=[0,1)$. 
	
	Furthermore $G_D$ is open. If not, then there exists an $f\in G_D$ sequence $(\epsilon_n)$, where $\epsilon_n \searrow 0$, and, for any $n$,  \mbox{ $\exists \ g_n\in B_{C^{\alpha}}(f, \epsilon_n)=\{ f\in \mathcal{E}^{\alpha}: \ \|g-f\|_{C^{\alpha}}< \epsilon_n \}$} and $K_n \in \mathcal{K}_D$, such that $g_n(K_n)=K_n$. Since $\mathcal{K}_D$ is compact, there exist a subsequence $(K_{n_k})_{k\in \mathbb{N}}$ and a $D$-invariant set $S$, such that $K_{n_k}\overset{d_H}{\longrightarrow}S$. Moreover $f(S)=S$. Indeed, since $g_{n_k}(K_{n_k}) = K_{n_k}$, 
	$$d_H\big(S, f(S)\big) \leq d_H\big(S,K_{n_k}\big) + d_H\big(f(K_{n_k}),g(K_{n_k})\big)$$
	and $d_H\big(f(K_{n_k}),g_{n_k}(K_{n_k})\big) \to 0$, since $g_n \overset{C^{\alpha}}{\longrightarrow} f$, which is a contradiction since we assumed that $f\in G_D$.
\end{proof}

\begin{note}
Theorem \ref{Th. locally result} can be similarly be stated and proven when instead of the doubling map, $D$, we consider any map $f\in \mathcal{E}^{\alpha}$.
\end{note}

\smallskip

\section{Dimensional results}

In this section we will study invariant sets in more detail.  In particular we follow the observation that if $K$ is an invariant set for the orientation preserving maps $f, g$ that are topologically conjugated $h\circ f=g \circ h$ via a homeomorphism $h$, then $h$ has to map an endpoint to an endpoint since $h$ is monotone as a homeomorphism. Also observe that the investigation of orientation preserving maps is sufficient for our purposes in order to draw results for all monotone expanding Markov maps. In other words we may also assume that $f$ is orientation preserving and thus $f' > 0$ and $\chi_f$ is increasing, otherwise we consider the square of the map. That leads, in a natural way, to the investigation of endpoints of invariant sets and the size of the invariant sets contained between a left and a right endpoint. Of course all the results of this section are not limited to only the case of monotone expanding Markov maps.

The set $M:=\{\underline{x}\in \Sigma_{m+1} : \ \sigma^n(\underline{x}) < \underline{x} , \ \forall n>0\}$ contains all right endpoints of all the invariant sets and we have that $\dim_H (M) = 1$, when we consider the doubling map and its respective usual coding space, \cite{Johan BAN}; in particular when $m=2$. Following the exact same reasoning as in the proof there, we will show that this is true for every $m\in\mathbb{N}$.

\begin{lemma} \label{Lemma dimM=1 for all m}
The set $M:=\{\underline{x}\in \Sigma_{m+1} : \ \sigma^n(\underline{x}) < \underline{x} , \ \forall n>0\}$ has full Hausdorff dimension for every $m\in \mathbb{N}$.
\end{lemma}
\begin{proof}
By Theorem \ref{Parry 2}, the set $M$ is the set of $\beta$ expansions of $1$, for $1< \beta <m$. Let $\beta_k$ be the real number related to $1^k0^{\infty}$, i.e. the unique $\beta_k \in (1, m)$ such that $1^k0^{\infty} = 1_{\beta_k}$ (Theorem \ref{Parry 2}).
Then we have that $I(\beta_{k-1}, \beta_k) \subset M$ (Proposition \ref{Parry 0}). 

By Theorem \ref{Schmeling} the numbers, $\beta_{k-1}$, $\beta_k$ can be arbitrarily close for large enough $k$. We also get the following inequalities,
\begin{align*}
\dim_H(M)
 &\geq \ 
 \dim_H\bigg(\rho_{\beta_{k-1}, \beta_k}^{-1}\big( [\beta_{k-1}, \beta_k]\big)\bigg) \geq
 \\&\geq \ 
 \frac{\ln \beta_{k-1}}{\ln \beta_k} \dim_H\big( [\beta_{k-1}, \beta_k] \big) 
 = \ 
 \frac{\ln \beta_{k-1}}{\ln \beta_k}
\end{align*}
where the right hand side of the inequality can grow arbitrarily close to 1, for sufficiently large $k>1$.
\end{proof}

We clearly have that $M_{\underline{c}} = M_{\underline{c}}' \ \cap \ M$ and thus $\dim_H(M_{\underline{c}}) \leq \dim_H (M_{\underline{c}}')$. The next result shows that the Hausdorff dimension of those two sets are in fact equal. Observe that for $\underline{c}=m^{\infty}$ we have the result from Lemma \ref{Lemma dimM=1 for all m}.

\begin{lemma} \label{Lemma "minimal" kneading} 
Let $\underline{c}=(c_1, c_2, c_3, \ldots ) \in \Sigma_{m+1}$, so that it is not a kneading sequence and $\underline{c}\neq m^{\infty}$. Then there exists a kneading sequence $\underline{d}$ such that $\underline{c} < \underline{d}$ and if $A:=\{\underline{a} \in \Sigma_{m+1} : \ \underline{a} \text{ is another kneading sequence so that } \underline{c} < \underline{a} < \underline{d} \}$, then $A$ is empty.
\end{lemma}

\begin{terminology}
Let $\underline{x}=(x_1, x_2, \ldots ) \in \Sigma_{m+1}$. A \textit{$x_1$-block}, $B$, is a block $[x_1x_1\ldots x_1]$ that appears in $\underline{x}$ in the form $[yx_1x_1\ldots x_1z]=yBz$, where $y,z\neq x_1$. Of course, if the length of $B$ is denoted by $r \in \mathbb{N}$, then $r\in [1, \infty]$. In particular, the \textit{the first $x_1$-block}, $B_1$, is the first biggest block of the form $[x_1x_1\ldots x_1]$ that appears in $\underline{x}$ and if the length of $B_1$ is denoted by $r_1 \in \mathbb{N}$, then $r_1\in [1, \infty]$ and $\underline{x}= (x_1)^{\infty}$, if $r_1=\infty$ \ or \ $\underline{x}= B_1x_{r_1+1}x_{r_1+2}\ldots$, otherwise.
\end{terminology}

\begin{proof}
Set $\ell_0 := \min\{ \ell\in \mathbb{N}: \ \sigma^{\ell}(\underline{c})\geq \underline{c} \}$. Firstly we consider the case $\ell_0=1$, i.e. $\sigma(\underline{c}) \geq \underline{c}$. Observe that $c_1$ cannot be equal to $m$, otherwise $\underline{c}=m^{\infty}$. Indeed, if $c_1=m$, the relation $\sigma(\underline{c})\geq \underline{c}$ implies that $c_2\geq c_1=m$. Follows that $c_3\geq c_2=m$, from which follows that $c_4\geq c_3=m$ and so on. In other words, we would have that $c=m^{\infty}$.
We set  $$\underline{d}= (c_1+1)0^{\infty}$$
Then clearly $\underline{d}>\underline{c}$ and it is kneading. Let now $\underline{a}$ be a kneading sequence so that $\underline{c} < \underline{a}$. Then $a_1\geq c_1$ and in fact we will show that $a_1 > c_1$, which completes the proof for $\ell_0=1$. Assume that $a_1=c_1$. Clearly $\sigma(\underline{c}) \geq \underline{c}$ implies that $c_2\geq c_1$ and since $\underline{a}>\underline{c}$, we get that $a_2\geq c_2 \geq c_1=a_1$.  But, since $\underline{a}$ is kneading, $a_1\ge a_2$ and thus $a_1=a_2=c_1$. Repeating the same process we  get that $a_3=a_2=a_1=c_1$ so that we finally get that $\underline{a} = c_1^{\infty}$, which is not kneading. 
Thus, indeed, $a_1>c_1$ or equivalently $a_1\geq c_1+1$, which implies that $\underline{a} \geq \underline{d}$. Thus in this case, $A$ is the empty set.

Now we consider the case $\ell_0 > 1$. Then $\underline{c}\neq i^{\infty}$. Also, from minimality of $\ell_0$, $c_1 \geq c_2,c_3, \ldots , c_{\ell_0}$ and the first $c_1$-block has greater or equal length from any other $c_1$-block that appears in the block $[c_1c_2\ldots c_{\ell_0}]$. 
\begin{itemize}
\item Assume that $\sigma^{\ell_0}(\underline{c})= \underline{c}$. Then $\underline{c}$ is periodic with period $\ell_0$ (from minimality of $\ell_0$). In other words $\underline{c}$ is of the form $(c_1c_2\ldots c_{\ell_0})^{\infty}$ where $c_1\ge c_2, c_3, \ldots , c_{\ell_0}$  (again from minimality of $\ell_0$). In particular $c_{\ell_0}< c_1$. Indeed, since $\underline{c}\neq i^{\infty}$, we have that there exists an $i\in \{ 2,3,\ldots , \ell_0 \}$ so that $c_i < c_1$. Set $i_0:= \min\{i: \ c_i<c_1  \}$. Then $\underline{c}=\underbrace{c_1c_1\ldots c_1}_{(i_0-1)\text{-times}}c_{i_0}\ldots c_{\ell_0} \ \ldots = (c_1)^{i_0-1} c_{i_0}\ldots c_{\ell_0} \ \ldots$. If $c_{\ell_0}\ge c_1$ we have that,
\begin{align*}
\sigma^{\ell_0 -1}(\underline{c}) \ 
&= \ 
c_{\ell_0}(c_1)^{i_0-2}\underset{\tiny \substack{\uparrow \\i_0\text{-position}}}{c_1}c_{i_0}\ldots  c_{\ell_0} \ \ldots \ 
\\ \overset{c_{\ell_0}\geq c_1}&{\ge} \ 
c_{1}(c_1)^{i_0-2}\underset{\tiny {\substack{\uparrow \\i_0\text{-position}}}}{c_1} c_{i_0} \ldots c_{\ell_0} \ \ldots \ 
\\ \overset{c_1>c_{i_0}}&{>} \ 
(c_1)^{i_0-1}\underset{\tiny {\substack{\uparrow \\i_0\text{-position}}}}{c_{i_0}} m^{\infty} \ 
\\&\geq \ 
(c_1)^{i_0-1} c_{i_0}\ldots c_{\ell_0} \ \ldots \ 
= \ 
\underline{c}
\end{align*} 
which contradicts the minimality of $\ell_0$. 
We set 
$\underline{d} = c_1 c_2 \ldots (c_{\ell_0}+1)0^{\infty}$. 
Then clearly $\underline{d} > \underline{c}$. Also we cannot have that another kneading sequence in between for it would be of the form 
$$\underline{a}= (c_1c_2\ldots c_{\ell_0})^k c_1c_2\ldots c_{i-1} c_{i}' \ a_{k\cdot \ell_0 +i +1} \ a_{k\cdot \ell_0+i +2} \ \ldots$$
for some $k\geq 1$, since $\underline{a}< \underline{d}$ with $c_i' > c_i$, $i\in {1,2, \ldots , \ell_0}$. But then, 
$$\sigma^{k\cdot \ell_0-1}(\underline{a}) \ 
= \ 
c_1\ldots c_{i-1}c_i' \ a_{k\cdot \ell_0+i +1} \ a_{k\cdot \ell_0+i +2} \ \ldots \ 
> \ 
\underline{a}$$
Thus, taking for granted that $\underline{d}$ is kneading we have that $A=\emptyset$.

We are left with showing that $\underline{d}$ is indeed kneading. It is obvious when $c_1 > c_i$, for all $i\in \{2,\ldots , \ell_0\}$. If not, from minimality of $\ell_0$, either the first $c_1$-block, $B_1$, is of strictly greater length than any other of the $c_1$-blocks, $B$, that appear before $\ell_0$, or there exists another $c_1$-block, $B_j$, $j>1$, of the same length, where though (from minimality of $\ell_0$), there exists a block $C$ and $c'>c''$ so that they appear in $\underline{c}$ as blocks of the form $B_1Cc'$ and $B_jCc''$.
In both cases we get that $\underline{d}$ is  kneading.

\item Assume now that $\sigma^{\ell_0}(\underline{c})> \underline{c}$  Then there exists a $k\geq 1$ so that $c_{\ell_0 + k}> c_{k}$ and set $k_0$ be the minimum of those $k$'s. Then, we have that $c_{\ell_0} < c_1$. Indeed, 
\begin{itemize}
\item if $k_0=1$, then $c_{\ell_0+1}>c_1$ and if we assume that $c_{\ell_0}\geq c_1$,
\begin{align*}
\sigma^{\ell_0-1}(\underline{c}) \ 
= \ 
c_{\ell_0}c_{\ell_0 + 1}\ldots \ 
\geq \ 
c_1c_{\ell_0 + 1}\ldots \ 
\overset{c_{\ell_0+1}>c_1 \geq c_2}{>}
c_1c_2m^{\infty} \ 
\geq \ 
\underline{c}
\end{align*}
which contradicts the minimality of $\ell_0$.

\item if $k_0>1$, then $c_{\ell_0+i}=c_{i}$, for all $1\leq i<k_0$. Now, if the length of the first $c_1$-block is $r\geq 1$, we have that $r<\ell_0 + k_0$. Indeed, since otherwise we would have that $c_{\ell+k_0} = c_{k_0} = c_1$ , which contradicts the definition of $k_0$. In fact we have that $r<\ell_0 + k_0-1$, for if $r=\ell_0 + k_0-1$, then $\underline{c}$ would be of the form 
$$\underbrace{c_1 \ldots c_1}_{\ell_0+k_0-1 \text{-times}} c_{\ell_0+k_0}$$
and $c_{\ell_0+k_0}$ must be strictly greater than $c_1$. But then we would have that $\sigma(\underline{c}) > \underline{c}$, which contradicts the minimality of $\ell_0$.

This means that there exist a $p< k_0-1$ so that $c_{\ell_0 +p} \neq c_1$ and by the minimality of $k_0$, this means that $c_{\ell_0 +p} < c_1$. We set $p_0$ to be the first such $p$.  Since  $c_{\ell_0+i}=c_{i}$, for all $1\leq i<k_0$, we have that $c_{p_0}= c_{\ell_0 + p_0}< c_1$ and also that $p_0$ is the first term so that $c_{p_0} < c_1$. In other words $p_0= r + 1$ (which also means that $r < k_0$).

If we assume that $c_{\ell_0}\ge c_{1}$, then,
\begin{align*}
\sigma^{\ell_0-1}(\underline{c}) \  
&= \  
c_{\ell_0}c_{\ell_0 + 1}\ldots c_{\ell_0 +p_0}\ldots c_{\ell_0+k_0-1}c_{\ell_0+k_0} \ \ldots \ 
\\ \overset {\tiny \substack{c_i=c_{\ell_0+ i} \\ \forall 1\leq i\leq k_0-1} }&{=} \ 
c_{\ell_0}c_{1}\ldots c_{p_0}\ldots c_{k_0-1}c_{\ell_0+k_0} \ \ldots \ 
\\ \overset{c_{\ell_0} \geq c_1}&{\geq} \ 
c_1c_{ 1}\ldots c_{p_0}\ldots c_{k_0-1}c_{\ell_0+k_0} \ \ldots \ 
\\&= \ 
(c_1)^{r+1}\ldots c_{p_0}\ldots c_{k_0-1}c_{\ell_0+k_0} \ \ldots \ 
\\&> \ 
(c_1)^{r}\ldots c_{p_0}\ldots c_{k_0-1}c_{\ell_0+k_0} \ \ldots \  
\\&> \ 
(c_1)^{r}\ldots c_{p_0}\ldots c_{k_0-1}c_{k_0} m^{\infty} \  
\\&> \
\underline{c}
\end{align*}
which cannot hold from the definition of $\ell_0$.
\end{itemize}

We set 
$$\underline{d} := c_1c_2\ldots (c_{\ell_0}+1)0^{\infty}  .$$
Clearly $\underline{d}> \underline{c}$ and we cannot consider any other sequence in between, otherwise it would have to be of the form
$$\underline{a}=c_1\ldots c_{\ell_0}a_{\ell_0+1}\ldots$$
where for some $i\geq 1$, $a_{\ell_0+i}>c_{i}$ (and $a_{\ell_0+j}\geq c_j$, for all $1\leq j < i$) . But, from the inequality, $\sigma^{\ell_0}(\underline{c})> \underline{c}$, we would have that $\sigma^{\ell_0}(\underline{a})> \underline{a}$ and thus $\underline{a}$ cannot be kneading.

We are left with showing that $\underline{d}$ is kneading, where the same argument as above works in this case as well and that completes the proof. 
\end{itemize}
\end{proof}

\begin{proposition}\label{Pr. dim M_c = dim M_c'} 
Let $\underline{c}=(c_1, c_2, c_3, \ldots ) \in \Sigma_{m+1}$, then $\dim_H(M_{\underline{c}}) = \dim_H (M_{\underline{c}}')$.
\end{proposition}
\begin{proof}  We firstly assume that $\underline{c}$ is a kneading sequence, i.e. $\sigma^n(\underline{c}) < \underline{c}$ for every $n>0$. Under this assumption we have that $\dim_H(M_{\underline{c}}') = \dim_H(M_{\underline{c}}'')$, where \hbox{$M_{\underline{c}}'' :=\{ \underline{x}\in \Sigma_{m+1} : \ \sigma^n(\underline{x}) < \underline{c} , \ \forall n > 0 \}$.} Indeed, let $\underline{x}\in M_{\underline{c}}$ be such that $\exists n_0=n_0(\underline{x}) \in \mathbb{N}$ for which we have $\sigma^{n_0}(\underline{x})=\underline{c}$. Then $n_0$ is unique, since for every $n>n_0$, $\sigma^{n}(\underline{x}) = \sigma^{n-n_0}(\sigma^{n_0}(\underline{x})) = \sigma^{n-n_0}(\underline{c})< \underline{c}$, where the last (strict) inequality holds since $\underline{c}$ is a kneading sequence. Thus such a sequence is of the form $B_{n_0}\underline{c}$ where $B_{n_0}$ is a block of length $n_0$ (also called $n$-block) and $B_{n_0}\leq \underline{c}$. Therefore, 
$$\{\underline{x}\in M_{\underline{c}}: \ \exists n_0=n_0(\underline{x}) \text{ such that } \sigma^{n_0}(\underline{x})=\underline{c}\}
\subseteq
\bigcup_{n=1}^{\infty} \{ B_n \underline{c}: \ B_n \text{ is an } n \text{-block}\}$$
and the last set is countable as countable union of finite sets. Thus the Hausdorff dimension is not affected by those points, i.e.,
$$\dim_H(M_{\underline{c}}') = \dim_H(M_{\underline{c}}'') \ .$$

Observe that $M_{\underline{c}}'' = S_{\beta_{\underline{c}}}$, where $\beta_{\underline{c}}$ is the unique real number in $(1,m]$ such that $\underline{c}=1_{\beta_{\underline{c}}}$. Our strategy in order to complete the proof for this case is to consider the encoding of the hole interval by $S_{\beta_{\underline{c}}}$. 
Following the same reasoning as in Lemma \ref{Lemma dimM=1 for all m}, since $M_{\underline{c}}\subset S_{\beta_{\underline{c}}}$, we will show that $M_{\underline{c}}$ has full Hausdorff dimension in $S_{\beta_{\underline{c}}}$, 
i.e. $\dim_{H}(M_{\underline{c}} \subset S_{\beta_{\underline{c}}}) \  = \  1$.
In this setting, $M_{\underline{c}}$ is the set of all $\beta$-expansions of $1$ for $1 < \beta < \beta_{\underline{c}}$. 
Observe now that if $\underline{c}=10^{\infty}$, then $M_{\underline{c}}=\{ \underline{c}\}$ and $M_{\underline{c}}' = \{ 0^{\infty}, \underline{c} \}$. In particular the two sets have the same Hausdorff dimension. So we are looking for $\underline{c}>10^{\infty}$. This means that either $c_1>1$ or $c_1=1$ and there is an $\ell\geq 2$ so that $c_{\ell}=1$ ($c_{\ell}$ cannot be greater than $1$, since $\underline{c}$ is considered to be kneading) and $c_{i}=0$ for every $1<i<\ell$, i.e. $\ell$ is the first position that we hit a non-zero term after $c_1$. 
Then for any $k\in \mathbb{N}$ we consider the kneading sequence $\underline{d}_{k}$ as follows,
$$ \underline{d}_{k}:= \ 
\begin{cases}
(c_1-1)^k0^{\infty}, &\quad c_1 > 1 \\
(10)^k0^{\infty}, &\quad c_1=c_2=1 \ (\ell=2) \\
(10^{\ell-2}0)^k0^{\infty}, &\quad c_1=1 \ \& \ \ell>2
\end{cases}
$$

Let $\beta_k$ be the real number related to $\underline{d}_{k}$, i.e. the unique $\beta_k \in (1, \beta_{\underline{c}})$ such that $\underline{d}_{k} = 1_{\beta_k}$ (Theorem \ref{Parry 2}).
Then we have that $I(\beta_{k-1}, \beta_k) \subset M_{\underline{c}}$ (Proposition \ref{Parry 0}). 

By Theorem \ref{Schmeling} the numbers, $\beta_{k-1}$, $\beta_k$ can be arbitrarily close for large enough $k$. We also get the following inequalities,
\begin{align*}
\dim_H(M_{\underline{c}}) \ 
&\geq \ 
\dim_H\bigg(\rho_{\beta_{k-1}, \beta_k}^{-1}\big( [\beta_{k-1}, \beta_k]\big)\bigg) \geq
\\&\geq \ 
\frac{\ln \beta_{k-1}}{\ln \beta_k} \dim_H\big( [\beta_{k-1}, \beta_k] \big) 
= \  
\frac{\ln \beta_{k-1}}{\ln \beta_k}
\end{align*}
where the right hand side of the inequality can grow arbitrarily close to 1, for sufficiently large $k>1$.

For the general case, we assume that $\underline{c}\in \Sigma_{m+1}$ is not a kneading sequence. If we consider the sequence $\underline{d}$, then by Lemma \ref{Lemma "minimal" kneading} we have that,
$$M_{\underline{c}} = \{\underline{x}: \ \sigma^n(\underline{x})< \underline{x}\leq \underline{d}, \ \forall n\geq 1 \} \ = M_{\underline{d}} \ .$$
In particular,
$$\dim_H (M_{\underline{c}}) = \dim_H (M_{\underline{d}}) \ .$$
Then, as we proved above, $M_{\underline{d}}$ has full Hausdorff dimension in $S_{\beta_{\underline{d}}}$, where $\beta_{\underline{d}}$ is the unique real number in $(1,m]$ such that $\underline{d}=1_{\beta_{\underline{d}}}$. In particular, 
$$\dim_H(M_{\underline{c}}) = \dim_H(M_{\underline{d}}')$$
Now since $M_{\underline{d}}' \supset M_{\underline{c}}' \supset M_{\underline{c}}$, we get the result.
\end{proof}

\bigskip

A similar approach works for a more general case. Namely, consider the sets 
\begin{align*}
&M_{\underline{c}, \underline{d}}':= \{ \underline{x}\in \Sigma_{m+1} : \ \underline{c} \leq  \sigma^n(\underline{x}) \leq \underline{d} , \ \forall n\geq 0 \} \\&
M_{\underline{c}, \underline{d}}:= \{ \underline{x}\in \Sigma_{m+1} : \ \underline{c} \leq\sigma^n(\underline{x})< x \leq \underline{d} , \ \forall n\geq 1\} \ .
\end{align*} 
Then $M_{\underline{c}, \underline{d}} \subset M_{\underline{c}, \underline{d}}'$ and thus $\dim_H(M_{\underline{c}, \underline{d}}) \leq \dim_H(M_{\underline{c}, \underline{d}}')$. In particular, we have that they have the same Hausdorff dimension.

\begin{proposition}\label{Pr. dim M_c,d = dim M_c,d'} 
Let $\underline{c}=(c_1, c_2, c_3, \ldots )$, $\underline{d}=(d_1, d_2, d_3, \ldots) \in \Sigma_{m+1}$ with $\underline{c}< \underline{d}$, then $\dim_H(M_{\underline{c}, \underline{d}}) = \dim_H (M_{\underline{c}, \underline{d}}')$
\end{proposition}
\begin{proof}
Let us firstly assume that $\underline{c}$ and $\underline{d}$ are both kneading sequences.
Then, exactly as in Proposition \ref{Pr. dim M_c = dim M_c'}, we can show that the sets $M_{\underline{c}, \underline{d}}'$ and \hbox{$M_{\underline{c}, \underline{d}}'':= \{ \underline{x}\in \Sigma_{m+1} : \ \underline{c} \leq  \sigma^n(\underline{x}) < \underline{d} , \ \forall n\geq 0 \}$} have the same Hausdorff dimension. Furthermore, $M_{\underline{c}, \underline{d}} \subset M_{\underline{c}, \underline{d}}''$ and if $\beta_{\underline{c}}$ and $\beta_{\underline{d}}$ are the real numbers in $(1,m]$ so that $1_{\beta_{\underline{c}}} = \underline{c}$ and $1_{\beta_{\underline{d}}} = \underline{d}$, then $M_{\underline{c}, \underline{d}}'' \subset S_{\beta_{\underline{d}}}\setminus S_{\beta_{\underline{c}}}$. Thus it is sufficient to show that $\dim_H \left( M_{\underline{c}, \underline{d}} \subset S_{\beta_{\underline{d}}} \setminus S_{\beta_{\underline{c}}} \right) = 1$.

Observe that $d_1> c_1$, for if $d_1=c_1$, then $\underline{d}< (d_1)^{\infty}$. Thus if $\underline{x}\in M_{\underline{c}, \underline{d}}$ then $\underline{x}<(d_1)^{\infty}$ which implies that $\exists i\in \mathbb{N}$ so that $x_i<d_1$ and thus $\sigma^{i-1}(\underline{x}) < \underline{c}$ which is a contradiction
(in particular if $\underline{x} \in M_{\underline{c}, \underline{d}}$ then $x_1>c_1$ and $x_i\geq c_1$ for all $i\geq 2$). 
In other words $M_{\underline{c}, \underline{d}} = \emptyset$ in this case. Firstly, if $d_1 \geq c_1+2$, then consider, 
$$\underline{a}_k := (c_1+1)^k(c_1)^{\infty}$$
for any $k\in \mathbb{N}$. Then $(\underline{a}_k) \subset M_{\underline{c}, \underline{d}}$ and it is a Cauchy sequence.

We assume now that $d_1=c_1+1$. Since $\underline{d}$ is assumed to be kneading, we have that $d_i\le c_1+1$ for all $i\geq 2$ and there exists an $i_0\geq 2$ so that the inequality is strict. If $\underline{d}=(c_1+1)(c_1)^{\infty}$ then as above we can show that $M_{\underline{c}, \underline{d}}= \emptyset$. Therefore there exists an $j_0\geq 2$ so that $d_{j_0}=c_1+1$. We consider two cases:
\begin{itemize}
\item $\underline{d}=(c_1+1)^{\ell_0}D$, for some $\ell_0 \geq 2$ where $D$ is an infinite block (so that $\underline{d}$ is kneading). Then consider,
$$\underline{a}_k := (c_1+1)^{\ell_0-1}(c_1)^k (c_1+1)(c_1)^{\infty}$$
for any $k\in \mathbb{N}$. Then $(\underline{a}_k) \subset M_{\underline{c}, \underline{d}}$ and it is a Cauchy sequence.
\item $\underline{d}=(c_1+1)d_2 d_3 \ldots d_{\ell} (c_1+1) D''$ where $d_1, \ldots , d_{\ell} \leq c_1$ (since and $D''$ is an infinite block (so that $\underline{d}$ is kneading). In particular, from the above, $(c_1+1)c_1^{\ell} (c_1+1)D''=\underline{d}$. Then consider,
$$\underline{a}_k := (c_1+1) (c_1)^{\ell+k}(c_1+1)(c_1)^{\infty}$$
for any $k\in \mathbb{N}$. Then $(\underline{a}_k) \subset M_{\underline{c}, \underline{d}}$ and it is a Cauchy sequence.
\end{itemize}

Let $\beta_k$ be the real number related to $\underline{a}_{k}$, i.e. the unique $\beta_k \in (\beta_{\underline{c}}, \beta_{\underline{d}})$ such that \hbox{$\underline{a}_{k} = 1_{\beta_k}$} (Theorem \ref{Parry 2}).
Then we have that $I(\beta_{k-1}, \beta_k) \subset M_{\underline{c}}$ (Proposition \ref{Parry 0}). 
Furthermore, observe that $(\underline{a}_k)$ is Cauchy in $\Sigma_{m+1}$ and by Theorem \ref{Schmeling} the numbers, $\beta_{k-1}$, $\beta_k$ can be arbitrarily close for large enough $k$. We also get the following inequalities,
\begin{align*}
\dim_H(M_{\underline{c}, \underline{d}}) 
&\geq \ 
\dim_H\bigg(\rho_{\beta_{k-1}, \beta_k}^{-1}\big( [\beta_{k-1}, \beta_k]\big)\bigg) \ 
\\&\geq \ 
\frac{\ln \beta_{k-1}}{\ln \beta_k} \dim_H\big( [\beta_{k-1}, \beta_k] \big) \ 
= \ 
\frac{\ln \beta_{k-1}}{\ln \beta_k}
\end{align*}
In particular, 
$$\dim_H \left( M_{\underline{c}, \underline{d}} \subset S_{\beta_{\underline{d}}} \setminus S_{\beta_{\underline{c}}} \right) = 1 \ .$$

For the general case, i.e. when $\underline{c}$ and $\underline{d}$ are not necessarily kneading. In case $\underline{d}$ is not kneading, by Lemma \ref{Lemma "minimal" kneading}, we consider the kneading sequence $\underline{d}'$ so that $\underline{d} < \underline{d}'$ and there is no other kneading sequence in between. Thus, 
$$\dim_H(M_{\underline{c}, \underline{d}}) \ = \  \dim_H (M_{\underline{c}, \underline{d}'}) \ .$$
If $(\underline{c}_n)$ is a sequence of kneading sequences that converge to $\underline{c}$ then we can choose a subsequence so that $\underline{c}_{k_n}\nearrow \underline{c}$. Then 
by Theorem \ref{Th. upper box dim = Hausdorff dim = entropy} and Lemma \ref{Lemma semicontinuity of entropy} 
we have that $\dim_H (M_{\underline{c}_{k_n}, \underline{d}'}) \to \dim_H (M_{\underline{c}, \underline{d}'})$ and $\dim_H (M_{\underline{c}_{k_n}, \underline{d}'}') \to \dim_H (M_{\underline{c}, \underline{d}'}')$. From the above, since $\underline{c}_{k_n}$ and $\underline{d}'$ are kneading, $\dim_H (M_{\underline{c}_{k_n}, \underline{d}'}) = \dim_H (M_{\underline{c}_{k_n}, \underline{d}'}')$ for any $k_n$. Thus $\dim_H (M_{\underline{c}, \underline{d}'}) = \dim_H (M_{\underline{c}, \underline{d}'}')$ and in particular, $\dim_H (M_{\underline{c}, \underline{d}}) = \dim_H (M_{\underline{c}, \underline{d}'}')$. Since $M_{\underline{c}, \underline{d}'}' \supset M_{\underline{c}, \underline{d}}'$, we finally get that,
$$\dim_H (M_{\underline{c}, \underline{d}}) = \dim_H (M_{\underline{c}, \underline{d}}') \ .$$

If on the other hand, there is not such a sequence $(\underline{c}_n)$, then there exists a kneading sequence $\underline{c}' < \underline{c}$ so that there is no other kneading sequence in between. Then obviously $M_{\underline{c}, \underline{d}} = M_{\underline{c}', \underline{d}'}$ and thus $\dim_H(M_{\underline{c}, \underline{d}}) = \dim_H (M_{\underline{c}', \underline{d}'}')$. Now since $M_{\underline{c}, \underline{d}}' \subset M_{\underline{c}', \underline{d}'}'$ we get the requested, i.e.
$$\dim_H (M_{\underline{c}, \underline{d}}) = \dim_H (M_{\underline{c}, \underline{d}}') \ .$$
\end{proof}

\smallskip

\section{Critical Points}
Examining a little bit further the bahaviour off those sets, we can find some conditions for when the Hausodrff dimension is zero. In fact we can describe the "critical points" so that $M_{\underline{c}, \underline{d}}$ has dimension zero.

Observe that if $\underline{x}\in M_{\underline{c}, \underline{d}}$, then  $\underline{c} \leq \sigma^{\ell}(\underline{x}) < \underline{x} \leq \underline{d}$, for all $\ell \geq 1$, which means that,
$$ a_1 \leq x_i \leq b_1 \quad \forall i\geq 1$$
In other words, the remaining allowed alphabet is $\{c_1, c_1+1, \ldots , d_1-1, d_1 \}$. In particular, if $b_1 = a_1$ then, $M_{\underline{c}, \underline{d}} = \emptyset$.

In other words, the remaining allowed alphabet is $\{c_1, c_1+1, \ldots , d_1-1, d_1 \}$. In particular, if $c_1 = d_1$ then, $M_{\underline{c}, \underline{d}}' = \emptyset$. We examine now the case where $c_1 +1= d_1$. If $d_2\leq c_1$ then $M_{\underline{c}, \underline{d}} = \emptyset$. If $d_2 > c_1$, then $M_{\underline{c}, \underline{d}}$ has a positive Hausdorff dimension. Indeed, let $\underline{c}=c_1c_2\ldots$ and $\underline{d}=(c_1+1)d_2\ldots$, with $d_1>c_1$. Since we assume that $\underline{d}$ is kneading, then $d_2=d_1=c_1+1$. Then $M_{\underline{c}, \underline{d}}$ contains the subshift of finite type, 
$$S_1 := \{ \underline{x}\in \{c_1, c_1+1\}^{\mathbb{N}}| \ x_i=d_1=c_1+1 \Rightarrow x_{i+1}=c_1\}$$
where $S_1$ has positive entropy and thus positive Hausdorff dimension.

In general, if we have that all $d_i\leq c_1$, for $i>1$ then $\dim_H(M_{\underline{c},\underline{d}})=0$. In particular if $\underline{d}=(c_1+1)c_1^{\infty}$, the largest sequence so that $d_i\leq c_1$, for $i>1$, then $\dim_H(M_{\underline{c},\underline{d}})=0$. So we have to look even further in order to have positive dimension. In fact this is exactly the critical point, i.e. for any $\underline{d}\gneqq (c_1+1)c_1^{\infty}$, we have that $\dim(M_{\underline{c},\underline{d}}) > 0$, since it contains a subshift of finite type of the form 
$$S_n := \{ \underline{x}\in \{c_1, c_1+1\}^{\mathbb{N}}| \ x_i=d_1=c_1+1 \Rightarrow x_{i+1}=x_{i+2}=\ldots = x_{i+n}=c_1\}$$
which has a positive entropy.

We sum up the above discussion in the following Proposition.
\begin{proposition}
Let $\underline{c} < \underline{d}$. 
If we assume that $\underline{d}$ is kneading then, 
\begin{align*}
	\dim_H(M_{\underline{c},\underline{d}}) = 0, \quad \text{if } \underline{d}\leq (c_1+1)c_1^{\infty}\\
 	\dim_H(M_{\underline{c},\underline{d}}) > 0, \quad \text{if } \underline{d}> (c_1+1)c_1^{\infty}
 \end{align*}
\end{proposition}

\begin{remark*}
	From Lemma \ref{Lemma "minimal" kneading} the case of $\underline{d}$ being a kneading sequence is sufficient.
\end{remark*}

\smallskip

\section{Continuity} \label{Section Continuity}
In \cite{Johan BAN}, J. Nilsson proves that the function $\phi : \underline{c}\mapsto \dim_H(\widetilde{M_{\underline{\widetilde{c}}}'})$ (see Example \ref{example of inversion}) is continuous for the case where $f(x)=2x (mod1)$. Of course we get it for $f\in \mathcal{E}^{\alpha}$, where $f$ is of order $2$. 
We will prove that this is still true for a map in $\mathcal{E}^{\alpha}$ of any order $m\in \mathbb{N}$. In fact we will prove the corresponding continuity property for $M_{\underline{c}, \underline{d}}'$, i.e. that $\dim_H (M_{\underline{c}, \underline{d}}')$ depends continuously on the pair $(\underline{c}, \underline{d})$, which gives the continuity of $\underline{c}\mapsto M_{\underline{c}'}$ as a corollary.

If $\underline{c}$ ends with $m^{\infty}$, then every term of a sequence of sequences approaching $\underline{c}$ from above is eventually equal to $\underline{c}$. Thus we will not consider this 'trivial' case as it can be derived in a same way as the more complicated case where $\underline{c}$ does not end with $m^{\infty}$. Therefore, we can assume that there exists a strictly increasing sequence $(j_n) \subset \mathbb{N}$ so that, $c_{j_n}<m$.

For each $n \in \mathbb{N}$, consider the following sequences,
\begin{itemize}
\item $\underline{y}_n := c_1\ldots (c_{j_n}+1)0^{\infty}$
\item $\underline{z}_n := d_1\ldots d_{j_n}0^{\infty}$
\end{itemize}
Then  for all $n \in \mathbb{N}$,  $\underline{c} \leq \underline{y}_n$
and
$\underline{z}_n \leq \underline{d}$. 
In particular, $\underline{y}_n \searrow \underline{c}$ and $\underline{z}_n \nearrow \underline{d}$. 

Let $B_{k}(S)$ denote the allowed blocks of length $k$ for the subshift $S \subset \Sigma_{m+1}$. For any sufficiently large $n$ so that $k <  j_{n}$, we have that, 
\begin{equation} \label{eq. allowed blocks of same length from inside for large n}
B_{k}(M_{\underline{c}, \underline{d}}') = B_{k}(M_{\underline{y}_{n} , \underline{z}_{n}}')
\end{equation}

Indeed, for $k < j_n$, we have that if the block $[x_1\ldots x_k]$ is in $B_{k}(M_{\underline{c}, \underline{d}}')$, then, as mentioned above, $x_1$, $x_2$, \ldots, $x_k \in \{ c_1, c_1+1, \ldots , d_{1}\}$ and it does not contain subblocks of the form
\begin{itemize}
\item $[c_1(c_2-1)] , \ [c_1(c_2-2)] , \ \ldots \ , \ [c_1c_1]$, if $c_2 \geq c_1$\\
and \\
no further restrictions for the subblocks of length greater or equal to $2$ other than those that arise from the remaining allowed alphabet, if $c_2 < c_1$.

\item $[c_1c_2(c_3-1)] , \ [c_1c_2(c_3-2)] , \ \ldots \ , \ [c_1c_2c_1]$, if $c_3 \geq c_1$\\
and\\
no further restrictions for the subblocks of length greater or equal to $3$ other than those that arise from the remaining allowed alphabet, if $c_3 < c_1$.
\end{itemize}
$$\vdots$$

\begin{itemize}
\item $[c_1 c_2 \ldots c_{k-1}(c_{k}-1)] , \ [c_1 c_2 \ldots c_{k-1}(c_{k}-2)] , \ \ldots \ , \ [c_1 c_2 \ldots c_{k-1}c_1]$, if $c_k \geq c_1$\\
and\\
no further restrictions for the subblocks of length greater or equal to $k$ other than those that arise from the remaining allowed alphabet, if $c_k < c_1$.

and

\item $[d_1(d_{2}+1)] , \ [d_1(d_{2}+2)] , \ \ldots \ , \ [d_1 d_1] $, if $d_2 \leq d_1$\\
and\\
no further restrictions for the subblocks of length greater or equal to $2$ other than those that arise from the remaining allowed alphabet, if $d_2 > d_1$.

\item $[d_1d_{2}(d_{3}+1)] , \ [d_1d_{2}(d_{3}+2)] , \ \ldots \ , \ [d_1d_{2}d_1]$, if $d_3 \leq d_1$\\
and\\
no further restrictions for the subblocks of length greater or equal to $3$ other than those that arise from the remaining allowed alphabet, if $d_3 > d_1$.
\end{itemize}
$$\vdots$$
\begin{itemize}
\item  $d_1d_{2}\ldots d_{k-1}(d_{k}+1) , \ d_1 d_{2} \ldots d_{k-1}(d_{k}+2) , \ \ldots  , \ d_1d_{2}\ldots d_{k-1}d_1$, if $d_k \leq d_1 $\\
and\\
no further restrictions for the subblocks of length greater or equal to $k$ other than those that arise from the remaining allowed alphabet, if $d_k > d_1$
\end{itemize}
Hence, the possible restrictions come from the first $k$-block of the sequences $\underline{c}$ and $\underline{d}$. Since $k< j_n$, the first $k$-blocks of $\underline{y}_n$ and $\underline{d}_n$ coincide with those of $\underline{c}$ and $\underline{d}$ respectively, so exactly the same subblocks are excluded for $B_{k} (M_{\underline{y}_n , \underline{z}_n}')$.

\begin{lemma} (Application of Fatou) \label{Lemma application of Fatou on N}
Let $(a_{n,\ell})_{n, \ell}$ be a sequence so that
\begin{itemize}
\item $a_{n, \ell} > 0$, for all $n$, $\ell \in \mathbb{N}$,
\item it is increasing w.r.t. $n$, in the sense that $a_{n, \ell} \leq a_{n, \ell+1}$,
\item the limits $\lim\limits_{\ell \to \infty}a_{n, \ell}$ and $\lim\limits_{n\to \infty} \lim\limits_{\ell\to \infty} a_{n, \ell}$ exist.
\end{itemize}
Then,
$$\lim_{n \to \infty} \lim_{\ell \to \infty} a_{n, \ell} \geq \lim_{\ell \to \infty} \lim_{n \to \infty} a_{n, \ell}$$
\end{lemma}
\begin{proof}
Consider the sequence $b_{n,1}:= a_{n, 1}$ and $b_{n, \ell}:= a_{n, \ell}-a_{n, \ell-1}$, for $\ell>1$ and the counting measure $\mu$ on the measurable space $(\mathbb{N}, \mathcal{P}(\mathbb{N}))$. Then by Fatou's Lemma,
$$\liminf_{n} \int_{\mathbb{N}} b_{n, \ell} \ d\mu (\ell) \ \geq \ \int_{\mathbb{N}} \liminf_{n} b_{n, \ell} \ d\mu (\ell) \ .$$
Since we integrate w.r.t. the counting measure and from the definition of $b_{n, \ell}$
\begin{align*}
LHS
=
\liminf_n \sum_{\ell =1}^{\infty} b_{n, \ell}
=
\liminf_n \lim_{L\to \infty} \sum_{\ell =1}^{L} b_{n, \ell}
=
\liminf_n \lim_{L\to \infty} a_{n, L}
=
\lim_n \lim_{L\to \infty} a_{n, L}
\end{align*}
Furthermore, taking also into account that $a_{n, \ell}$ is increasing w.r.t. $n$, we have that,
\begin{align*}
RHS
&=
\sum_{\ell=1}^{\infty} \liminf_{n} b_{n, \ell}
= 
\lim_{L\to \infty} \sum_{\ell=1}^{L} \liminf_{n} b_{n, \ell}
\\&=
\lim_{L\to \infty} \liminf_{n} \sum_{\ell=1}^{L} b_{n, \ell}
=
\lim_{L\to \infty} \liminf_{n} a_{n, L}
=
\lim_{L\to \infty} \lim_{n} a_{n, L}
\end{align*}
\end{proof}

\begin{lemma} \label{Lemma most important part for continuity}
Let $\underline{c}$, $\underline{d}$ be two sequences in $\Sigma_{m+1}$. Furthermore, consider for every $n\in \mathbb{N}$ the sequences $\underline{y}_{n}$ and $\underline{z}_{n}$ as above. Then,
$$h_{top}(M_{\underline{c}, \underline{d}}') 
= 
\lim_{n \to \infty} h_{top}(M_{\underline{y}_n, \underline{z}_{n}}')
$$
\end{lemma}
\begin{proof}
We set $a_{n, \ell}:= \frac{1}{\ell} \log(|B_{\ell}(M_{\underline{y}_n , \underline{z}_n}')|)$. Then by Lemma \ref{Lemma application of Fatou on N} and relation (\ref{eq. allowed blocks of same length from inside for large n}) we have that,
\begin{align*}
h_{top}(M_{\underline{c}, \underline{d}}') \quad 
&\geq \ 
\lim_{n \to \infty} h_{top}(M_{\underline{y}_n, \underline{z}_n}') \ 
\\&= \ 
\lim_{n \to \infty} \lim_{\ell \to \infty} \frac{1}{\ell} \log \big( |B_{\ell}(M_{\underline{y}_{n}, \underline{z}_{n}}')| \big)  
\\&\geq  \ 
\lim_{\ell \to \infty} \lim_{n \to \infty} \frac{1}{\ell} \log \big( |B_{\ell}(M_{\underline{y}_{n}, \underline{z}_{n}}')| \big)
\\&= \ 
\lim_{\ell \to \infty} \lim_{n \to \infty} \frac{1}{\ell} \log \big( |B_{\ell}(M_{\underline{c}, \underline{d}}')| \big)
\\&= \ 
\lim_{\ell \to \infty} \frac{1}{\ell} \log \big( |B_{\ell}(M_{\underline{c}, \underline{d}}')| \big)
 \ = \quad 
h_{top}(M_{\underline{c}, \underline{d}}')
\end{align*}
\end{proof}

With the above Lemma in hand, we are now ready to prove that the dimension of $M_{\underline{c}, \underline{d}}$ depends continuously on the pair $(\underline{c}, \underline{d})$.

\begin{theorem} \label{Th. continuity of Phi}
Let the map defined by $\Phi: (\underline{c}, \underline{d}) \mapsto \dim_{H}(M_{\underline{c}, \underline{d}}')$.  The map $\Phi$ is continuous.
\end{theorem}
\begin{proof}
It is sufficient to show that there exist sequences $\underline{w}_n \nearrow \underline{c}$, $\underline{x}_n \searrow \underline{d}$ and $\underline{y}_n \searrow \underline{c}$, $\underline{z}_n \nearrow \underline{d}$, such that,
$$\lim_{n \to \infty} \Phi(\underline{w}_n, \underline{x}_n) \ = \ \Phi(\underline{c}, \underline{d}) \ = \ \lim_{n \to \infty} \Phi(\underline{y}_n, \underline{z}_n)$$
In that case, both of the limits exists from monotonicity. 
Using Theorem \ref{Th. upper box dim = Hausdorff dim = entropy} and Lemma \ref{Lemma semicontinuity of entropy}, we have that the first equality indeed holds for any pair of sequences so that $\underline{w}_n \nearrow \underline{c}$ and $\underline{x}_n \searrow \underline{d}$, since then $M_{\underline{c}, \underline{d}}' = \bigcap_{n} M_{\underline{w}_n, \underline{x}_n}'$ and $M_{\underline{w}_{n+1}, \underline{x}_{n+1}}' \subset M_{\underline{w}_n, \underline{x}_n}'$. 

The second equality is given from Theorem \ref{Th. upper box dim = Hausdorff dim = entropy} and 
Lemma \ref{Lemma most important part for continuity} and that completes the proof.
\end{proof}

\begin{remark*}
The first inequality in the proof of Theorem \ref{Th. continuity of Phi} can also be derived in a similar manner as the second equality.
\end{remark*}

\subsection{Corollaries of Theorem \ref{Th. continuity of Phi}}

\begin{theorem}
The functions $\phi: \underline{c} \mapsto \dim_H(M_{\underline{c}}')$ and $\widetilde{\phi}: \underline{c} \mapsto \dim_H(\widetilde{M_{\underline{\widetilde{c}}}}')$ are continuous.
\end{theorem}

Let $N_{\underline{c}}:= \{ \underline{x}\in \Sigma_{m+1}: \ \underline{c} \leq \sigma^n(\underline{x}) < \underline{x} \}$. By Proposition \ref{Pr. dim M_c,d = dim M_c,d'} we also have the following.
 
\begin{theorem} \label{Th. continuity of Psi}
The map defined by $\Psi: (\underline{c}, \underline{d}) \mapsto \dim_{H}(M_{\underline{c}, \underline{d}})$ is continuous. In particular, the functions $\psi: \underline{c} \mapsto \dim_H(M_{\underline{c}})$ and $\widetilde{\psi}: \underline{c} \mapsto \dim_H(N_{\underline{c}})$ are continuous.
\end{theorem}

\smallskip

\section{Further Discussion}

Firstly, in view of Lemma \ref{Lemma "minimal" kneading}, we will prove the following Lemma.

\begin{lemma} \label{Lemma region of a kneading sequence as a left endpoint} 
Let $\underline{c}=(c_1, c_2, c_3, \ldots ) \in \Sigma_{m+1}$, be a kneading sequence. Then there exists an $\epsilon = \epsilon(\underline{c}) > 0$ so that if $\underline{c}'$ is a sequence 
so that 
if $d_{\Sigma_{m+1}}(\underline{c}, \underline{c'})<\epsilon$, then 
$M_{\underline{c}, \underline{d}} = M_{\underline{c}', \underline{d}}$ and  $M_{\underline{c}, \underline{d}}' = M_{\underline{c}', \underline{d}}'$. In particular, $M_{\underline{c}, \underline{d}}' = \{ \underline{x}\in \Sigma_{m+1} : \ \underline{c} <  \sigma^n(\underline{x}) \leq \underline{d} , \ \forall n\geq 0 \}$ and $M_{\underline{c}, \underline{d}} = \{ \underline{x}\in \Sigma_{m+1} : \ \underline{c} < \sigma^n(\underline{x})< x \leq \underline{d} , \ \forall n\geq 1\}$
\end{lemma}
\begin{proof}
Let $\underline{c}$ be a kneading sequence. Then we consider $c_{min}:=\min \{c_i: \ i\in \mathbb{N} \}$ and $i_0$ the first time $c_{min}$ appears in the sequence $\underline{c}$. Since $\underline{c}$ is kneading, observe that $c_{min}< c_1$ since $\underline{c}\neq n^{\infty}$, $n\in \{0,1,\ldots , m\}$ and of course $i_0>1$. If $\underline{c}'$ is such that $\underline{c}'=c_1' c_2' \ldots$, so that $c_i'=c_i$, for all $1\leq i \leq i_0$ and $\underline{c}'< \underline{c}$, then for every $\underline{c}''$ in between we have that $c_{i_0}'' = c_{i_0}<c_1=c_1''$ and thus $\sigma^{i_0-1}(\underline{c}'')< \underline{c}'$. In particular, if we assume that $\underline{x}\in M_{\underline{c}', \underline{d}} \setminus M_{\underline{c}, \underline{d}}$, then there exists an $\ell\geq 0$ so that $\underline{c}' \leq \sigma^{\ell}(\underline{x}) < \underline{c}$. But then $\sigma^{\ell+i_0-1}(\underline{x})< \underline{c}'$, which is a contradiction, since we assumed that $\underline{x} \in M_{\underline{c}', \underline{d}}$. Thus $M_{\underline{c}', \underline{d}} \setminus M_{\underline{c}, \underline{d}} = \emptyset$.

By symmetry if $\underline{c}'$ is such that $\underline{c}'=c_1' c_2' \ldots$, so that $c_i'=c_i$, for all $1\leq i \leq i_0$ and $\underline{c}'> \underline{c}$, then for every $\underline{c}''$ in between we have that $c_{i_0}'' = c_{i_0}<c_1=c_1''$ and thus $\sigma^{i_0-1}(\underline{c}'')< \underline{c}$ and therefore,  $M_{\underline{c}, \underline{d}} \setminus M_{\underline{c}', \underline{d}} = \emptyset$.
\end{proof}

Lemma \ref{Lemma "minimal" kneading} and Lemma \ref{Lemma region of a kneading sequence as a left endpoint} give us an idea of what the critical points for the general case should be, i.e. describe the smallest interval so that we hit for the first time any value for the dimension for the first time  or the largest interval so that the dimension remains unchanged. For example, if $\dim_H(M_{\underline{c}, \underline{d}}) = a \in [0,1]$ and if $\underline{c}$ is a kneading sequence, the dimension remains the same if we perturb a little bit $\underline{c}$.  On the other end, if $\underline{d}$ is not a kneading sequence, then the dimension is not changed if we enlarge the interval $[\underline{c}, \underline{d}]$ by a little bit from the right end. This is one reason that we are interested in understanding the behaviour of the dimension in the particular case when $\underline{c}$ is not a kneading sequence and $\underline{d}$ is a kneading one.

Another reason is that from Theorem \ref{Th. continuity of Phi} and Theorem \ref{Th. continuity of Psi}, we can, in principle, approximate the dimension for of $M_{\underline{c}, \underline{d}}$  for any given interval $[\underline{c}, \underline{d}]$, by calculating the dimension of $M_{\underline{c}', \underline{d}'}$ where $\underline{c}'$ is a not kneading sequence close to $\underline{c}$ and $\underline{d}'$ is a kneading sequence close to $\underline{d}$.  As a matter of fact we can even assume that $\underline{c}'$ and  $\underline{d}'$ are even finite sequences. Indeed, for $\underline{d}'$ we have that either it is finite, i.e. ends with infinite $0$'s or it can be approached from below or above by an increasing sequence of finite kneading sequences. On the other hand, either $\underline{c}'$ is finite or it can be approached, either from below or from above, by a sequence of finite non-kneading sequences. This reduces the problem of estimating the Hausdorff dimension to calculate the dimension of a SFT. That can be easily derived by a similar approach as in the discussion at the beginning of Section \ref{Section Continuity}. Therefore it is sufficient to calculate the dimension of a subshift described above, i.e. find the largest eigenvalue of the adjacency matrix. In principle, this estimation can be arbitrarily sharp.

\bigskip

\begin{ackn}
The author would like to express his gratitude to J\"org Schmeling, under whose supervision and guidance this project was conducted and to Tomas Persson for his helpful comments and recommendations.
\end{ackn}

\bigskip

\footnotesize \textsc{Centre for Mathematical Sciences, Lund University, 221 00 Lund, Sweden.}

\textit{E-mail address}: \href{mailto:georgios.lamprinakis@math.lth.se}{\nolinkurl{georgios.lamprinakis@math.lth.se}}

\end{document}